\documentclass[a4paper]{article}
\usepackage{amsthm,amssymb,amsmath,enumerate,graphicx,psfrag,color}
\usepackage{tikz}

\usepackage{hyperref}
\hypersetup{
 colorlinks=true,  
 linkcolor=black,    
 citecolor=black,    
 urlcolor=black      
}

\usetikzlibrary{backgrounds,positioning}
\usetikzlibrary{shapes.geometric}

\newtheorem{definition}{Definition}
\newtheorem{claim}{Claim}

\newtheorem{theorem}[definition]{Theorem}

\newtheorem{lemma}[definition]{Lemma}

\newcommand{\bigO}{\ensuremath{O}}
\newcommand{\bigOmega}{\ensuremath{\Omega}}

\newcommand{\comment}[1]{}
\newcommand{\N}{\mathbb N}

\newcommand{\es}{\emptyset}

\newcommand{\cP}{\mathcal{P}}

\newcommand{\cF}{\mathcal{F}}

\newcommand{\cH}{\mathcal{H}}

\newcommand{\cE}{\mathcal{E}}

\newcommand{\cG}{\mathcal{G}}
\newcommand{\cJ}{\mathcal{J}}

\newcommand{\clos}[1]{\overline{#1}}

\tikzstyle{hvertex}=[thick,circle,inner sep=0.cm, minimum size=2.2mm, fill=white, draw=black]
\tikzstyle{smallvx}=[thick,circle,inner sep=0.cm, minimum size=1.5mm, fill=white, draw=black]
\tikzstyle{hedge}=[very thick]
\colorlet{hellgrau}{black!30!white}
\colorlet{dunkelgrau}{black!50!white}
\colorlet{hellblau}{blue!20!white}
\colorlet{hellrot}{red!40!white}
\tikzstyle{convcols}=[color=white,fill=hellblau]
\usetikzlibrary{calc}

\newcommand{\EP}{Erd\H{o}s-P\'osa property}

\newcommand{\dist}{{\rm dist}}

\newcommand{\ds}{{\rm ds}}

\newcommand{\emtext}[1]{\text{\em #1}}

\newcommand{\sm}{\setminus}

\newcommand{\new}[1]{#1}
\newenvironment{newe}{\color{red}}{}
\newcommand{\bn}{\begin{newe}}
\newcommand{\en}{\end{newe}}

\newcommand{\len}{\ell}

\title{Long cycles have the edge-Erd\H{o}s-P\'osa property} 
\author{Henning Bruhn, Matthias Heinlein and Felix Joos\thanks{The research was also supported by the EPSRC, grant no. EP/M009408/1.}}
\date{}

\begin{document}
\maketitle

\begin{abstract}
We prove that the set of long cycles has the edge-\EP:
for every fixed integer $\ell\ge 3$ and every $k\in\N$, every graph $G$ 
either contains $k$ edge-disjoint cycles of length at least $\ell$ (long cycles)
or an edge set $X$ of size $\bigO(k^2\log k +  \new{k\ell})$ such that $G-X$ does not contain any long cycle. 
This answers a question of Birmel{\'e}, Bondy, and Reed (Combinatorica 27 (2007), 135--145).
\end{abstract}

\section{Introduction}

Many theorems \new{in graph theory} have a vertex version and an edge version. 
There is a Menger theorem about (vertex-)disjoint paths and a variant
about edge-disjoint paths. We prove here the edge analogue of an Erd\H os-P\'osa-type
theorem.

Erd\H os and P\'osa~\cite{EP62} proved in 1962 
that every graph either contains $k$ disjoint cycles or a set of $\bigO(k\log k)$
vertices that meets every cycle. 
Since then many Erd\H os-P\'osa-type theorems have been discovered, among them one about 
\emph{long} cycles. These are cycles of a length that is at least some fixed integer~$\ell$. 

Indeed, every graph either contains $k$ disjoint long cycles or a 
set of $\bigO(k\ell+k\log k)$
vertices that meets every \new{long} cycle. With a worse bound this follows 
from a theorem of Robertson and Seymour~\cite{RS86}, while
the stated bound is due to Mousset, Noever, \new{\v{S}}kori\'c, and Weissenberger~\cite{MNSW16}. We prove an edge-disjoint analogue:
\begin{theorem}\label{thm:EPlongcycles}
Let $\ell$ be a  positive integer.
Then every graph $G$ either contains $k$ edge-disjoint long cycles
or a set $X\subseteq E(G)$ of size $\bigO(\new{k^2 \log k + k\ell })$ such that $G-X$ contains no long cycle.
\end{theorem}

This answers a question of Birmel{\'e}, Bondy, and Reed~\cite{BBR07}.

For 
vertex-disjoint long cycles, the bound of~$\bigO(k\ell+k\log k)$ proved by 
Mousset et al.~\cite{MNSW16} is optimal as it 
matches a lower bound found by Fiorini and Herinckx~\cite{FH14}.
We show below that the set $X$ in  Theorem~\ref{thm:EPlongcycles}
also needs to have size at least $\bigOmega(k\ell+k\log k)$. We believe that,
as in the vertex version, this is the right order of magnitude. 

A family $\cF$ of graphs has the \emph{\EP} 
if there is a function $f_\cF:\N\to\mathbb R$ such that 
for every integer $k$,  every graph $G$ either contains $k$ disjoint copies of graphs in $\cF$ or 
a \emph{hitting set} $X\subseteq V(G)$ of size at most $f_\cF(k)$ that meets every $\cF$-copy in $G$.
Thus cycles have the \EP, but also, for instance, even cycles~\cite{Tho88} and many other graph classes.

Many such results are the consequence of a far-reaching theorem of Robertson and Seymour~\cite{RS86}:
for a fixed graph $H$, the class of graphs that have $H$ as a minor has the \EP\
if and only if $H$ is planar. For example, the theorem implies that 
long cycles  have the \EP.

Less is known about the edge analogue of the \EP.
There, the objective is to find edge-disjoint copies of graphs in $\cF$
or a bounded hitting set of edges.  
While cycles have the \emph{edge-Erd\H os-P\'osa property}~\cite[Exercise 9.5]{Die10}, an edge version 
of  Robertson and Seymour's theorem, for example, is still wide open.
By our result, long cycles have the edge-\EP.

We know of only two other graph classes that have the 
edge-\EP: $S$-cycles, cycles that each contain a vertex from a fixed set $S$,
and the graphs that contain a $\theta_r$-minor, where $\theta_r$ is the multigraph 
consisting of two vertices linked by $r$ parallel edges. 
The first result is due to Pontecorvi and Wollan~\cite{PW12}, the second due to Raymond, Sau 
and Thilikos~\cite{RST13}. 
Strikingly, both results are obtained via a reduction to their respective vertex versions.
For long cycles this does not seem to 
be possible (at least not that easily), and consequently, our proof is direct.

Within restricted ambient graphs, two more graph classes are known 
to have the edge-\EP. 
Odd cycles do not have the \EP, and they do not have the edge 
version either~\cite{DNL87}. The same is true for the class of graphs
that contain an immersion\footnote{
A graph $G$ contains an immersion of $H$
if there is an injective function $\tau:V(H)\to V(G)$
and edge-disjoint $\tau(u)$--$\tau(v)$-paths for every $uv\in E(H)$ in $G$.
} of~$H$ for certain graphs~$H$.
If, however, the ambient graphs $G$ are required to be $4$-edge-connected, then 
odd cycles as well as graphs with an $H$-immersion 
gain the edge-\EP~\cite{KK16,Liu15}.

\medskip
There are many more results about the ordinary \EP, most of which 
are listed in the survey of 
Raymond and Thilikos~\cite{RT16}.
A direction we find interesting concerns rooted graphs. In this setting, a set $S$ (or two 
or more such sets) is fixed in the ambient graph $G$. The target objects
are required to meet the set $S$ in some specified way. For instance, 
$S$-cycles, cycles that each intersect $S$, have the \EP~\cite{KKM11,PW12},
and this is still true for long $S$-cycles~\cite{BJS14}.
Huynh, Joos, and Wollan~\cite{HJW16} verify the \EP\ for cycles
satisfying more general restrictions that include  for example  $S_1$-$S_2$-cycles 
\new {(cycles that intersect both $S_1$ and $S_2$)}.
Note that $S_1$-$S_2$-$S_3$-cycles do not have the \EP. 
We do not know whether the \EP\ extends to edge-disjoint $S_1$-$S_2$-cycles. 

\medskip

The Triangle Removal Lemma of Ruzsa and Szemer\'edi~\cite{RS78} also has a certain (edge-)\EP{} flavour.
Its many applications include, for example, a short proof of
Roth's celebrated result on $3$-term arithmetic progression in dense integer sets~\cite{Rot53}.
The lemma states that there is a function $f:(0,1)\to (0,1)$
such that \new{for} every graph $G$ on $n$ (sufficiently large) vertices \new{and every $0<\epsilon<1$}
either \new{$G$} contains $f(\epsilon) n^3$ 
(normally not edge-disjoint) triangles  or there is 
a set of edges $X\subseteq E(G)$
of size $|X|\leq \epsilon n^2$ such that $G-X$ is triangle-free. 
Analogous results are known for all graphs (instead of triangles) and even known in the uniform hypergraph setting (see~\cite{CF13} for a survey on this topic).
These results rely heavily on (hyper)graph regularity methods.
\medskip

In Section~\ref{sec:dis},
we discuss the size of the hitting set and how the \EP{} and its edge analogue differ.
In Section~\ref{sec:preliminaries},
we introduce tools needed in the proof of Theorem~\ref{thm:EPlongcycles}.
After a brief overview we  
prove Theorem~\ref{thm:EPlongcycles}
in Section~\ref{sec:mainproof}.

\section{Discussion}\label{sec:dis}

\subsection{The size of the hitting set}

Fiorini and Herinckx~\cite{FH14} observed that the hitting set for long cycles
in the ordinary \EP\ needs to have size at least $\bigOmega(k\ell+k\log k)$. 
That there is a hitting set of size $\bigO(k\ell+k\log k)$, the optimal size, 
is due to Mousset et al.~\cite{MNSW16} who built on earlier work of 
 Robertson and Seymour~\cite{RS86},
Birmel{\'e} et al.~\cite{BBR07}, and
Fiorini and Herinckx~\cite{FH14}.

What is the optimal size of the hitting set in the edge-disjoint version?
As for vertex-disjoint long cycles, the  construction of Simonovits~\cite{Sim67},
originally intended for the classic Erd\H os-P\'osa theorem,
gives a lower bound of $\bigOmega(k\log k)$. Indeed, the graphs in 
the construction are cubic, which means that cycles are disjoint if and
only if they are edge-disjoint. 

That the size of the hitting set needs to depend on $\ell$ at all is not immediately 
obvious. But it does, and indeed, the dependence is linear. To prove this 
we construct graphs $S_\ell$ that do not contain two edge-disjoint long cycles
and that do not admit a hitting set of less than $\tfrac{\ell}{30}$ edges. 
Taking $k-1$ disjoint copies of $S_\ell$ then yields a graph 
without $k$ edge-disjoint \new{long} cycles and no hitting set of size smaller 
than $\tfrac{1}{30}(k-1)\ell=\bigOmega(k\ell)$. 
Therefore, the size of hitting sets for edge-disjoint long cycles
needs to be at least $\bigOmega(k\ell+k\log k)$.

\begin{figure}[htb]
\centering
\begin{tikzpicture}[scale=1]
\def\krad{1.2}
\def\angle{360/10}
\foreach \i in {0,...,9}{
  \begin{scope}[on background layer]
    \draw[hedge] (252+\i*\angle:\krad) -- (270+\i*\angle:1.5*\krad) -- (288+\i*\angle:\krad);    
  \end{scope}
  \foreach \j in {\i,...,9}{
  \begin{scope}[on background layer]
    \draw[hedge] (252+\i*\angle:\krad) -- (252+\j*\angle:\krad);
  \end{scope}
  \node[hvertex] (v\i) at (252+\i*\angle:\krad){};
  \node[hvertex] (w\i) at (270+\i*\angle:1.5*\krad){};
}}

\end{tikzpicture}
\caption{The graph $S_{17}$ contains no two edge-disjoint cycles of length at least~$17$.}\label{fig:lower}
\end{figure}
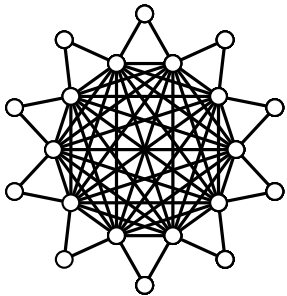

The graphs $S_\ell$ are constructed as follows.
Let $p=\lfloor \frac 2 3 (\ell-1) \rfloor$, and
let $S_\ell$ be the graph obtained from a clique on $p$ vertices $v_0, \ldots , v_{p-1}$
by adding vertices $w_0, \ldots , w_{p-1}$ such that each $w_i$ is adjacent to $v_{i-1}$ and $v_i$
(where we take indices mod~$p$).
The graphs $S_\ell$ are sometimes called \emph{suns}~\cite{ISGCI}.
As the clique contains only $p<\frac 2 3 \ell$ vertices, 
every long cycle  in $S_\ell$ passes
through at least \new{$\lceil\frac 1 3 \ell \rceil> \frac{p}{2}$} vertices of $\{w_0, \ldots, w_{p-1}\}$.
As these have degree~$2$, there cannot be two edge-disjoint long cycles in $S_\ell$.

Let $\ell\geq 30$, and 
consider any set $X$ of at most~$\frac{\ell}{30}$ edges. 
We show that $X$ is not a hitting set.
For every edge $uv\in X$, delete its endvertices $u$ and $v$ in $G$, 
and if we delete a vertex $v_i$ of the clique, also delete the adjacent vertices $w_{i}$ and $w_{i+1}$.
All in all, we delete a set $U$ of at most $6\cdot \tfrac{\ell}{30} \le \tfrac{\ell}{5} $ vertices in $G$.
For the cycle $C=v_0 \ldots v_{p-1} v_0$,
let $C_1, \ldots, C_r$ be the components of $C-U$.
Let $v_{s_i}$ and $v_{t_i}$ be the two endpoints of the \new{path} $C_i$.
None of the vertices $w_{s_i}, \ldots, w_{t_i-1}$ is deleted, and thus 
$P_i=v_{s_i}w_{s_i}v_{s_i+1}\ldots w_{t_i-1}v_{t_i}$ is a 
path  in $G-U$.

Concatenating the paths $P_i$  by adding the edges $v_{t_i}v_{s_{i+1}}$,
we obtain a Hamilton cycle $D$ of $G-U$. 
Noting that  $p\ge \frac 2 3 (\ell-3)$, 
we calculate that the length of $D$ is 
\[
|V(S_\ell)| - |U| =2p-\frac 1 5 \ell \ge \frac 4 3 (\ell-3) - \frac 1 5 \ell = \ell + \frac{2\ell - 60}{15} \ge \ell
\]
as $\ell\ge 30$.
Since $G-X\supseteq G-U$ still contains a long cycle, we deduce that
no edge set of size at most $\frac{\ell}{30}$ is a hitting set.
\medskip

Comparing the lower bound of $\bigOmega(k\ell + k\log k)$ with Theorem~\ref{thm:EPlongcycles},
we see that there is a gap in the second term by a factor $k$. 
We believe that the optimal size 
of the hitting set coincides with the lower bound. 

\subsection{Vertex versus edge version}

Why is the edge-\EP\ hard at all, especially when the corresponding 
vertex version is known? Cannot a reduction be employed or the proof
be adapted? Pontecorvi and Wollan~\cite{PW12} obtain the edge version 
for $S$-cycles from the vertex version by a simple gadget construction.
Essentially, they apply the vertex version to a modified line graph (a similar approach is also used by Kawarabayashi and Kobayashi~\cite{KK16}). 
Why is that not possible for long cycles?

Cycles do not have a unique image in the line graph. The line graph of a cycle 
is a cycle  but not every cycle in the line graph corresponds to a cycle 
in the root graph. The preimage of an $S$-cycle in the (slightly modified) line graph 
still contains an $S$-cycle---this is what allows Pontecorvi and Wollan 
to reduce to the vertex version. For long cycles this will not 
work because every cycle contained in the preimage of a long cycle might be short. 

So how about adapting the proof of the vertex version in some more or less
obvious way? While the existing proof might, and does in our case, 
give some clues, an easy adaption seems hopeless. We believe this is because
edge-disjoint long cycles actually require a mix of the two disjointness concepts.

Why is this?
For simplicity, consider the case $k=2$.
We could construct two long cycles in a graph $G$ as follows.
Choose $2\ell$ vertices $v_1,\ldots,v_\ell$ and $w_1,\ldots,w_\ell$.
For the vertex version, suppose that all these vertices are distinct.
What we now need to do is to find internally vertex-disjoint paths $P_1,\ldots,P_\ell$
 and $Q_1,\ldots,Q_\ell$
such that $P_i$  is a $v_i$--$v_{i+1}$-path and $Q_i$ a $w_i$--$w_{i+1}$-path for every $i=1,\ldots,\ell$
(where we set $v_{\ell+1}=v_1$ and $w_{\ell+1}=w_1$).
In the edge version, we only need to suppose that $v_i\neq v_j$ and $w_i\neq w_j$ for distinct $i,j$.
Again, we seek for paths connecting these vertices in cyclic order.
But, and that is the crucial point,
$P_i$ and $P_j$ as well as $Q_i$ and $Q_j$ need to be internally vertex-disjoint for distinct $i,j$,
while $P_i$ and $Q_j$ only need to be edge-disjoint.
That is, we  deal with two different types of disjointness.

If instead we  only require that all these paths are  edge-disjoint, 
then we obtain immersions of long cycles.
Strikingly, for immersions the adaption of vertex version arguments 
appears to work very well. Indeed, to prove his  strong result about edge-disjoint
immersions, 
Liu~\cite{Liu15} translates a part 
of the graph minor theory to line graphs. (The translation, however, 
is not at all trivial.)

\section{Preliminaries}\label{sec:preliminaries}

In this section we 
\new{introduce a number of tools and some notation. In particular, in Section~\ref{extsec}
we develop a tool for finding a short path (Lemma~\ref{extintlem}); in Section~\ref{framesec}, we treat
\emph{frames}, a structure that captures a good part of the long cycles in the graph; and 
finally in Section~\ref{sepsec}, we investigate how many edges are needed to split 
vertices of a given set into well-linked parts (Lemma~\ref{lem:gatessep2}).}
\new{The reader may find it useful to first skip parts of this section. Indeed, to start with the proof 
of the main theorem, to be found in Section~\ref{sec:mainproof},
 only the definition of a {frame}, in Section~\ref{framesec}, is needed. The rest of this section then 
can be consulted whenever necessary.}

All logarithms $\log n$ will be to base~$2$.
\new{If $G$ is a graph, and $F$ a subgraph of~$G$,
then $G-F$ denotes the graph obtained from $G$ by deleting all vertices of~$F$. 
In contrast, if  $Y\subseteq E(G)$ is an edge set then $G-Y$ is the graph obtained
from $G$ by deleting the edges in $Y$.
}

\subsection{Paths and cycles}
\label{subsec:bpc}
We follow the notation used in the textbook of Diestel~\cite{Die10}.
In particular, we write $P=u\ldots v$ for  a path $P$ with endvertices $u,v$ and say that $P$ is a $u$--$v$-path.
For two vertices $x,y\in V(P)$,
we denote by $xPy$ the subpath of $P$ with endvertices $x,y$.
For an oriented cycle $C$ and $x,y\in V(C)$, 
we also write $xCy$ 
to denote the $x$--$y$-subpath of $C$.
For paths $x_1P_1y_1$ and $x_2P_2y_2$
such that $x_2=y_1$ and otherwise $P_1$ and $P_2$ are disjoint,
we write $x_1P_1x_2P_2y_2$ for the concatenation of $P_1$ and $P_2$.
For two vertex sets $A,B$, we define an \emph{$A$--$B$-path} as a path $P$ such that one endpoint of $P$ lies in $A$ and one in $B$ and $P$ is internally disjoint from $A\cup B$.
For a subgraph $H$ of $G$ (or a vertex set which we treat as a subgraph without edges),
we define an \emph{$H$-path} as a path with two \new{distinct} 
endvertices in~$H$ that is internally disjoint from $H$. \new{A path of length~$1$ 
between two vertices of $H$ is only considered
to be an $H$-path if its single edge is not in $E(H)$.}

For a cycle $C$ and a path $P$, we denote by $\len(C)$ and $\len(P)$
the number of edges of $C$ and $P$, respectively, 
and refer to $\len(C)$ and $\len(P)$ as the \emph{length} of $C$ and $P$,
respectively.

Throughout the article, we fix a positive integer~$\ell$ and 
call  $P$ and $C$ \emph{short} if $\len(P)<\ell$ and $\len(C)<\ell$, respectively.
A cycle is called \emph{long} if its length is at least $\ell$.

\subsection{Extensions of paths}\label{extsec}

The key trick in our proofs is to exclude cycles of intermediate length, that 
is, cycles that are long but not too long. 
\new{That this is possible, is discussed in Section~\ref{sec:mainproof}. 
Later, in the proof of the main theorem, we exclude some cases by showing that 
	otherwise there would be intermediate cycles.
	This is often done by replacing a long path (for example as part of  a cycle) by a short one.
In this subsection we treat a tool that allows us to find such short paths. 
The only lemmas of this subsection that are used later are Lemmas~\ref{extintlem} and~\ref{singlejumplem}.
	}
\medskip

Consider a path $P$ with endvertices $u,v$.
We write $\leq_P$ for the total order of the vertices $V(P)$ induced by the distance from $u$ on $P$.
Let $Q_1,\ldots, Q_r$ be $P$-paths, and for $i=1,\ldots, r$, let  $u_i$ and $v_i$ be the  endvertices 
of $Q_i$ such that $u_i<_Pv_i$.
The tuple $(Q_1,\ldots,Q_r)$  is an \emph{extension} of $P$ if 
\begin{enumerate}[(E1)]
\item the paths $Q_1,\ldots, Q_r$ are pairwise internally disjoint; 
\item the cycle $u_iPv_i\cup Q_i$ is short for $i=1,\ldots, r$; 
\item $u_1=u$ and $v_r=v$;
\item \label{ext:order} $u_i <_P u_{i+1}<_P v_i <_P v_{i+1}$ for $i=1,\ldots, r-1$; and 
\item $v_i\leq_P u_{i+2}$ for $i=1,\ldots, r-2$.
\end{enumerate} 
See Figure~\ref{fig:extension} for an illustration.

\begin{figure}[htb]
\centering
\begin{tikzpicture}[label distance=-0.5mm,auto]

\tikzstyle{jumps}=[ultra thick, white, double distance=1pt, double=black,bend left=80]

\draw[hedge] (0.5,0) -- (7,0);

\node[smallvx, label=below:{\small $u_1$}] (u1) at (0.5,0){};
\node[smallvx, label=below:{\small $v_1=u_3$}] (v1) at (3,0){}; 
\node[smallvx, label=below:{\small $u_2$}] (u2) at (1.5,0){}; 
\node[smallvx, label=below:{\small $v_2$}] (v2) at (4,0){}; 
\node[smallvx, label=below:{\small $v_3$}] (v3) at (5.5,0){};
\node[smallvx, label=below:{\small $u_4$}] (u4) at (5,0){};
\node[smallvx, label=below:{\small $v_4$}] (v4) at (7,0){};

\draw[jumps] (u1) to (v1);
\draw[jumps] (u2) to (v2);
\draw[jumps] (v1) to (v3);
\draw[jumps] (u4) to (v4);
 
\node at (0.0,0) {$P$};
\node at (1.6,1.1) {$Q_1$};
\node at (2.9,1.1) {$Q_2$};
\node at (4.3,1.1) {$Q_3$};
\node at (6,1.1) {$Q_4$};

\end{tikzpicture}
\caption{A $P$-extension.}\label{fig:extension}
\end{figure}
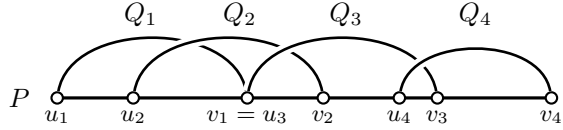

\begin{lemma}\label{lem:uniqueExtension}
Let $P$ be a path, and let $(Q_1,\ldots, Q_r)$ be an extension of $P$.
For any $i,j$ with $1\leq i\le j\leq r$, there is exactly one cycle $C$ 
in $P\cup\bigcup_{s=i}^j Q_s$ that contains $u_i,v_j$. The edge set of the cycle is
\begin{equation}\label{uniquecyc}
E(C)=E\big(P\cup\bigcup_{s=i}^j Q_s\big)\sm \bigcup_{t=i+1}^j E(u_tPv_{t-1}).
\end{equation}
\end{lemma}

\begin{proof}
The graph $H=P\cup\bigcup_{s=i}^j Q_s$ is $2$-connected as it is the union of 
cycles $u_sPv_s\cup Q_s$, such that consecutive cycles overlap in an edge. Thus
the graph $H$ contains a cycle~$C$ through $u_i$ and $v_j$.

Note that $C$ has to contain each of $Q_i,\ldots, Q_{j-1}$:
if $Q_t\nsubseteq C$ for a $t\in \{i,\ldots, j-1\}$ then, by (E5), $u_{t+1}$ separates $u_i$
and $v_j$ in $C$ , which is impossible. We also have $Q_j\subseteq C$ as otherwise~$v_j$
would have degree~$1$ in $C$ as $v_j\not\in Q_{j-1}$ by (E4). 

Now, for $t=i+1,\ldots, j$ the vertex $v_{t-1}$ has degree~$2$ in $C$. Therefore, 
either $u_tPv_{t-1}\subseteq C$ or $v_{t-1}Pu_{t+1}\subseteq C$
(where we temporarily interpret $u_{j+1}$ as $v_j$). 
However, $\{u_t,v_{t-1}\}$
separates $u_i$ from $v_j$ in $H$, which means that $C$ has to pass through \new{$u_i,u_t,v_j,v_{t-1}$ in this cyclic order.}
Thus $v_{t-1}Pu_{t+1}\subseteq C$ and $u_tPv_{t-1}\nsubseteq C$ (since already $Q_{t-1}\subseteq C$).
It is easy to check that this fixes $C$ to be as in~\eqref{uniquecyc}.
\end{proof}

\begin{lemma}\label{claim:cycleInExtensionIsShort}
Let $P$ be a path, and let $(Q_1,\ldots, Q_r)$ be an extension of $P$. 
Assume that every long cycle in $H=P\cup\bigcup_{s=1}^r Q_s$
has length at least~$2\ell$. 
Then every cycle in $H$ is short.
\end{lemma}
\begin{proof}
Suppose that $H$ contains a long cycle $C$.
Clearly, its intersection with $P$ is nonempty.
Let $i$ be the smallest index such that $u_i$ lies in $C$, and let $j$ 
be the largest index with $v_j\in V(C)$. Note that $i<j$ 
by the definition of extensions. 
We, furthermore, assume $C$
to be chosen such that $j-i$ is minimal.
Thus $C\subseteq u_iPv_j \cup \bigcup_{s=i}^j Q_s$.

The cycle $C$ satisfies the conditions of Lemma~\ref{lem:uniqueExtension},
which implies that its edge set is as in~\eqref{uniquecyc}. 
Let $C'$ be the unique cycle in $u_iPv_{j-1} \cup \bigcup_{s=i}^{j-1} Q_s$ containing $u_i$ and $v_{j-1}$.
Hence $C'$ is short by the choice of $C$, and its edge set is given by~\eqref{uniquecyc}---with $j-1$ instead of $j$. 
Then, $E(C)\Delta E(C')$ is equal to $u_jPv_j\cup Q_j$,
which is a short cycle by~(E2).
As $|E(C)|\leq |E(C')|+|E(C\Delta C')|< 2\ell$,
the length of the long cycle $C$ is less than~$2\ell$, which 
contradicts the assumption of the lemma.
\end{proof}

\new{We now prove the main lemma of the subsection.}
\begin{lemma}\label{extintlem}
Let $P$ be a path in a graph $G$, and let $(Q_1,\ldots, Q_r)$ be a tuple of $P$-paths 
that satisfy~{\rm(E2)--(E4)} and
\begin{enumerate}[\rm (E1$'$)]
\item if $|i-j|>1$, then $Q_i$ and $Q_j$ are internally disjoint.
\end{enumerate}
If every long cycle in $G$ has length at least $2\ell$, then there is 
\new{a short path between the endvertices of $P$ that is contained in 
$P\cup\bigcup_{j=1}^rQ_j$.}
\end{lemma}
\begin{proof}
	Among all tuples $(Q'_1, \ldots, Q'_s)$ of $P$-paths in $\bigcup_{j=1}^r Q_j$
	that satisfy (E2)--(E4) and {\rm (E1$'$)}, choose a tuple $T'=(Q'_1, \ldots, Q'_s)$ such that $s$ is minimal.
	Such a tuple exists as $(Q_1, \ldots, Q_r)$ satisfies (E2)--(E4) and {\rm (E1$'$)}.
	Let $u'_i$ and $v'_i$ be the endvertices of $Q'_i$ such that $u'_i<_P v'_i$ holds.
	
	Now, assume that there are two paths $Q'_i$ and $Q'_j$, $i<j$, that share an internal vertex.
	By {\rm (E1$'$)} we have $j=i+1$. Following $Q'_i$ from $u'_i$ on, 
	let $x$ be the first vertex of $Q'_i - u'_i$ that also belongs to $Q'_{i+1}$.
	Now define a new path $R$ as $R=u'_i Q'_i x Q'_{i+1} v'_{i+1}$.
	The path $R$ is a $P$-path as $x$ is an internal vertex and its endpoints are $u'_i$ and $v'_{i+1}$.
	Furthermore, the length of the cycle $R\cup u'_iPv'_{i+1}$ is at most 
	\[
	\ell(Q'_i\cup u'_iPv'_i) + \ell(Q'_{i+1}\cup u'_{i+1}Pv'_{i+1}) <  2\ell,
	\] 
	which implies that $R$ is short, by assumption.
	
	Now, the tuple $T''=(Q_1, \ldots, Q'_{i-1}, R, Q'_{i+2},\ldots Q'_s)$ satisfies (E2)--(E4) and {\rm (E1$'$)} 
	as (E2) was just proved, (E3) is trivial, 
	and (E4) and {\rm (E1$'$)} are inherited from $T'$ as $R$ just combines two consecutive paths of $T'$.
	However, $T''$ uses only $s-1$ paths,
	which contradicts the choice of $T'$. Thus, there are no such paths $Q'_i,Q'_j$ that share an internal vertex and hence $T'$ satisfies (E1).
	
	Assume, that $T'$ does not satisfy (E5); that is, there is an $i$ such that $u'_{i+2} <_P v'_i$.
	By (E4), we have $u'_i<_P u'_{i+1} <_P u'_{i+2}$ and $v'_i <_P v'_{i+1} <_P v'_{i+2}$ which implies 
	\[
		u'_i <_P u'_{i+2} <_P v'_i <_P v'_{i+2}.
	\]
	This is the statement of (E4) for the paths $Q'_i$ and $Q'_{i+2}$ which makes $Q_{i+1}$ unnecessary in $T$.
	This is again a contradiction to the minimality of $s$.
	Thus, the tuple $T'$ satisfies (E1)--(E5) and is therefore an extension of $P$.

\new{Now, Lemmas~\ref{lem:uniqueExtension}
and~\ref{claim:cycleInExtensionIsShort} imply that there is a short cycle in $P\cup\bigcup_{j=1}^rQ'_j$,
and thus in $P\cup\bigcup_{j=1}^rQ_j$,
through the endvertices of $P$. The cycle then contains the desired path.}
\end{proof}

\new{For later use, we prove a convenience lemma that helps constructing the paths $Q_i$ in Lemma~\ref{extintlem}.} 
\begin{lemma}\label{singlejumplem}
Let $P$ be a path in a graph $G$, and 
let  $C_1,\ldots, C_r$ be a set of short cycles such that 
\begin{enumerate}[\rm (i)]
\item $C_i\cap P=u_iPv_i$ for two (not necessarily distinct) vertices $u_i, v_i$, for $i=1,\ldots, r$;
\item $C_i$ and $C_{i+1}$ meet outside $P$  
for $i=1,\ldots, r-1$; and
\item $u_iPv_i$ and $u_{i+1}Pv_{i+1}$ meet for $i=1,\ldots, r-1$.
\end{enumerate}
If every long cycle in $G$ has length at least $3\ell$, then there is a 
short cycle $C\subseteq\bigcup_{i=1}^rC_i$
such that $C\cap P=u_1Pv_r$.
\end{lemma}
\begin{proof}
By induction on $r$ we show that: 
there is a 
short cycle $C\subseteq\bigcup_{i=1}^rC_i$
such that $C\cap P=u_1Pv_r$ and such that $C$ contains an edge 
in $E(C_r)\sm E(P)$ that is incident with $v_r$. 

The induction starts with $C=C_1$. Now, let $C'$ be such a cycle for $r-1$.
For every $i$, let $Q_i$ be the path $C_i-u_iPv_i$, and let $p_i$ and $q_i$
be its endvertices such that $p_i$ is a neighbour of $u_i$ in $C_i$ 
and $q_i$ a neighbour of $v_i$ in $C_i$. We define $Q'$ with endvertices $p',q'$
in the analogous way as $Q'=C'-u_1Pv_{r-1}$. 
 

Assume first that $C'$ and $C_r$ meet outside $P$.
Starting in $p'$ let $x$ be the first vertex in $Q'$ that lies in $Q_r$. 
Then put $C=u_1p'Q'xQ_rq_rv_r\cup u_1Pv_r$ and observe that $C$ satisfies all required properties
if, in addition, it is short. 
This holds, as $\ell(C)\le \ell(C')+(\ell(Q_r)+\ell(u_rPv_r)) \new{+ \ell(C_{r-1})} < \ell + \ell \new{+\ell}= \new{3}\ell$.

Next, assume that $Q'$ and $Q_r$ are disjoint outside $P$. 
Since the edge $q'v_{r-1}$ of $C'$ is an edge of $C_{r-1}$ we see that $q'\in V(Q_{r-1})$,
which means that $Q'$ and $Q_{r-1}$ have a vertex in common.
Starting from $p'$ let $y$ be the first vertex of $Q'$ that lies in $Q_{r-1}$.
Starting from $q_r$ let $z$ be the first vertex in $Q_r$ that lies in $Q_{r-1}$.
Since
$C_{r-1}$ and $C_r$ meet outside $P$, by~(ii), there is such a vertex $z$.
Put $C=u_1p'Q'yQ_{r-1}zQ_rq_rv_r \cup u_1Pv_r$ and observe that, again,  
$C$ satisfies all required properties
if it is short.

We now prove that $C$ is a short cycle. 
Using~(iii), we see that
\begin{align*}
\ell(C) &\leq \ell(C') + \ell(C_{r-1})+\ell(C_r)\\
&< \ell + \ell + \ell=3\ell, 
\end{align*}
as $C'$ is short by induction and as the other two 
terms are smaller than~$\ell$ as well. Thus, the length of the cycle $C$
is smaller than~$3\ell$, which means it is a short cycle.
\end{proof}

\subsection{Frames}\label{framesec}
Simonovits' short proof  of the Erd\H os-P\'osa theorem
rests on a \new{\emph{frame}, a }maximal subgraph of the ambient graph $G$, in which all 
the disjoint cycles are found~\cite{Sim67}. We mimic this approach
that also appears in other works~\cite{BJS14,PW12}.
However, in contrast to all such previous approaches, in our case this subgraph is not subcubic, but may have arbitrary high maximum degree.

\new{There is one more difference between our approach and that of Simonovits.
 In Simonovits' proof, there is a dichotomy: if 
the frame is \emph{large}, with respect to some appropriate measure, 
then there are $k$ disjoint cycles, 
and if the frame is \emph{small}
then it yields a hitting set. In our approach a large frame still contains $k$ edge-disjoint 
long cycles. 
A small frame, however, can still lead to both outcomes. We might find a hitting
edge set or we might find $k$ edge-disjoint cycles, but these will normally not be 
contained in the frame but also use parts outside the frame.
}

Any subgraph $F$ of a graph $G$ is a \emph{frame} of $G$ if its minimum degree~$\delta(F)$
is at least~$2$ and if every cycle in $F$ is long. 
For a frame $F$ of $G$, we define 
\begin{itemize}
\item $U(F)=\{v\in V(F) : d_F(v)\ge 3\}$, the set of vertices of degree at least~$3$ in~$F$; and
\item $\displaystyle \ds(F)=\sum_{u\in U(F)} d_F(u)$, the sum of the degrees of the vertices in $U(F)$.
\end{itemize}
In the proof we will choose a frame $F$ such that \new{$\ds(F)$ is maximal.}
The main motivation stems from the fact that large values in $\ds(F)$  yield $k$ edge-disjoint long cycles in $F$.
In the next lemma we collect a number of useful properties about frames.
\new{Lemma~\ref{framelem} is the only lemma of this subsection that is used later.}

\begin{lemma}\label{framelem}
\new{Let $G$ be a connected graph and let $F\subseteq G$ be a frame such that $\ds(F)$ is maximal.}
Then
\begin{enumerate}[\rm (i)]
\item $F$ is connected;
\item if $\ds(F)\geq 84k\log k$, then $G$ contains $k$ edge-disjoint long cycles;
\item every $F$-path is short; and
\item \label{uniqueShadow} 
there exists a short path $P=u\ldots v \subseteq F$ for every $F$-path $Q=u\ldots v$.
This path is unique if every long cycle in $G$ has length at least $2\ell$.
\end{enumerate}
\end{lemma}

\noindent We need some preparation before we can prove the lemma.

\begin{lemma}[Erd\H os and P\'osa {\cite{EP62}}]\label{lem:girth}
Let $G$ be a multigraph on $n$ vertices with $\delta(G)\geq 3$.
Then $G$ contains a cycle of length at most $\max\{2\log n,1\}$.
\end{lemma}

\begin{lemma} \label{lem:cyclesMultigraph}
Let $k\in \N$ and $G$ be a multigraph with  $|E(G)|\ge 42k\log k$ and $\delta(G)\ge 3$.
Then $G$ contains $k$ edge-disjoint cycles.
\end{lemma}
\begin{proof}
We proceed by induction on $k$.
For $k=1$ the statement holds, since every multigraph with $\delta(G)\ge 3$ contains a cycle.

Let $k\geq 2$. 
We may assume that $n\new{=|V(G)|}\geq 2$, as otherwise the statement is trivial.
Let $C$ be a shortest cycle in $G$.
Let $n_1$ and $n_2$ be the number of vertices of degree 1 and 2 in $G_0=G - E(C)$, respectively.
Thus $n_1+n_2\le \len(C)$.
As long as $G_t$ contains a vertex of degree $1$ or $2$,
let $G_{t+1}$ arise from $G_{t}$ by either deleting a vertex of degree~1
or suppressing a vertex of degree~$2$.
Let $s$ be the maximal integer for which $G_s$ is defined.
We claim that one of the following statements hold\new{s}
for the transformation from $G_t$ to $G_{t+1}$.
\begin{enumerate}[(i)]
\item The number of vertices of degree $1$ does not increase and the number of vertices of degree $2$ decreases.
\item The number of vertices of degree $1$ decreases and the number of vertices of degree $2$ increases by at most $1$.
\end{enumerate}
To see that our claim is true, 
suppose we deleted a vertex $u$ of degree $1$ and let $v$ be the neighbour of $u$.
If $d_{G_t}(v)=2$, then (i) holds and otherwise (ii) holds.
If we suppress a vertex of degree $2$, then (i) holds.

It is easy to see that (ii) holds at most $n_1$ times.
Hence (i) holds at most $n_1+n_2$ times.
Observe that $|E(G_{t})|=|E(G_{t+1})|-1$.
Therefore,
$|E(G_s)|\geq |E(G)|-\len(C)-2n_1-n_2 \geq |E(G)|-3\len(C)$.

Let $H$ arise from $G_s$ by deleting isolated vertices.
Thus
\begin{equation}\label{edgesH}
|E(H)| \geq |E(G)|-3\len(C).
\end{equation}
By construction, $H$ does not contain vertices of degree~$1$ or~$2$; 
thus, $\delta(H)\ge 3$ holds or $H$ is empty.
We claim that 
$|E(H)| > 42 (k-1)\log(k-1)\geq 0$. If true, $H$ contains in particular an edge,
which implies that $\delta(H)\ge 3$. Moreover, we can apply induction to $H$
to find $k-1$ edge-disjoint cycles in $H$. Since $G-E(C)$ contains a subdivision of $H$, 
we therefore obtain together with $C$ in total $k$ edge-disjoint cycles in $G$. 

It remains to prove that $|E(H)| > 42 (k-1)\log(k-1)$.
We write $m=|E(G)|$ and by $\delta(G)\ge 3$ we have $|V(G)|\le \frac{2m}{3}$.
As $C$ was chosen as the shortest cycle in $G$, Lemma \ref{lem:girth} implies
\begin{equation}\label{shortcyc}
\len(C)\le 2\log\left(\frac{2m}{3} \right).
\end{equation}
Note that the function $x\mapsto x-6\log\left(\frac 2 3 x\right)$ is increasing for $x\ge 9$.
Since $k\ge 2$, we conclude $\log(28\log k) \le 6\log k$. 
Together with $m\ge 42k\log k \ge 9$,  we deduce from~\eqref{edgesH} and~\eqref{shortcyc} that
\begin{align*}
|E(H)| &\ge m - 6\log\left(\frac 2 3 m\right) \\
&\ge 42k\log k - 6\log\left(28 k\log k\right) \\
&\ge 42k\log k - 6 \log k - 6\log(28\log k) \\
&\ge 42k\log k - 6\log k - 36\log k  \\
&> 42(k-1)\log(k-1).
\end{align*}
This finishes the proof.
\end{proof}

\begin{proof}[Proof of Lemma~\ref{framelem}]
For (i), suppose that $F$ has two components $A$ and $B$.
As $G$ is connected, 
there is an $A$--$B$-path $P$ in $G$ that is internally disjoint from $F$.
Thus, $F\cup P$ is a frame,
as $F\cup P$ contains the same cycles as $F$.
Since
$\ds(F\cup P)> \ds(F)$, we obtain a contradiction to the choice of $F$.

For (ii), denote by $H$ the multigraph obtained from $F$ by suppressing 
all vertices of degree~$2$. Observe that $|E(H)|=\tfrac{1}{2}\ds(F)\geq 42k\log k$ and $\delta(H)\ge 3$.
Thus, by Lemma \ref{lem:cyclesMultigraph}, $H$ and then also $F$ contain $k$ edge-disjoint cycles. Since all cycles
in $F$ are long, the assertion is proved.

For (iii), suppose there is a long $F$-path $Q$.
Then it can be added to $F$,
since in $F\cup Q$ all cycles are still long.
However, $\ds(F\cup Q)>\ds(F)$, which is a contradiction.

For (iv): 
As $F$ is connected by (i), the distance of $u$ and $v$ in $F$ is finite.
If $\dist_F(u,v)\geq \ell$, then any cycle 
in $F\cup Q$ containing $Q$ is long, which again contradicts~(iii) and proves the first part of (iv).
If there were two short $u$--$v$-paths $P_1,P_2$ in $F$, their union $P_1\cup P_2\subseteq F$ would contain a cycle of length less than $2\ell$
which is short by assumption. This is impossible as $F$ only contains long cycles.
\end{proof}

\subsection{Edge-connectivity}\label{sepsec}

The aim of this subsection is to prove Lemma~\ref{lem:gatessep2}, which 
helps defining a hitting set in Section~\ref{hittingset}.
We need Lemmas~\ref{lem:equi} and~\ref{lem:gatessep} only 
for the proof of Lemma~\ref{lem:gatessep2}.

Let $G$ be a multigraph and $k\in \N$.
For two vertices $u,v\in V(G)$,
we define $u\sim_k v$
if either $u=v$ or if there are $k$ edge-disjoint $u$--$v$-paths in $G$.
The transitivity of $\sim_k$ follows from Menger's theorem and thus $\sim_k$ is an equivalence relation.

\begin{lemma}\label{lem:equi}
Let $G$ be a multigraph and let $A,B$ be nonempty subsets of distinct 
equivalence classes of  $\sim_k$.
Then there is a set $X$ of at most $k-1$ edges separating~$A$ and~$B$.
\end{lemma}
\begin{proof}
Pick $a\in A$ and $b\in B$, and observe that $a\not\sim_k b$. 
Thus there is an edge set $X$ of size at most $k-1$ that separates~$a$ and~$b$ in~$G$.
Suppose that $X$ fails to separate $A$ from $B$ in $G$. Then there 
are $a'\in A$ and $b'\in B$
such that $G-X$ still contains an $a'$--$b'$-path.
Since $X$ is too small to separate $a$ from $a'$, and $b$ from $b'$,
we see that 
the vertices $a,a',b',b$ belong to the same component in $G-X$, which is a contradiction.
\end{proof}

\begin{lemma}\label{lem:gatessep}
Let $k,p\in \N$, and
let  $A_1,\ldots,A_p$
be subsets of~$p$ distinct equivalence classes of~$\sim_k$
in a multigraph~$G$.
Then there is an edge set $X\subseteq E(G)$ of size at most~$(p-1)(k-1)$
such that for all distinct $i,j\in \{1, \ldots, p\}$, the multigraph
$G-X$ does not contain any $A_i$--$A_j$-path.
\end{lemma}
\begin{proof}
Let $G'$ arise from $G$ by identifying for every $i\in\{1, \ldots, p\}$
 all vertices in $A_i$ to a single vertex $a_i$.
Any edge set of $G$ separates two distinct sets $A_i,A_j$  in $G$
if and only if, seen as an edge set of $G'$, it separates $a_i$ from $a_j$ in $G'$.
Since, by Lemma~\ref{lem:equi}, any two distinct sets $A_i,A_j$
can be separated by at most $k-1$ edges in $G$, also any two distinct $a_i,a_j$ 
can be separated by at most $k-1$ edges in $G'$.

We proceed by induction on $p$. If $p=1$, then $X=\emptyset$ will do. Thus, we may
assume that $p\geq 2$. By Lemma~\ref{lem:equi}, there is a set $X'$ of at most $k-1$
edges such that $a_1$ is separated from $a_2$ in $G'-X'$. 
Let $C_1,\ldots, C_s$ be the set of components of $G'-X'$ that 
contain at least one vertex of $a_1,\ldots, a_p$. Let $p_i$ be the number 
of vertices $a_1,\ldots, a_p$ in $C_i$. 
Then $p_i\geq 1$ but also $p_i<p$ for $i=1,\ldots,s$. Applying induction, we
obtain for each $i$ an edge set $X_i$ of size at most $(p_i-1)(k-1)$ 
such that no component of $C_i-X_i$ contains two or more of $a_1,\ldots, a_p$. 
Then $X=X'\cup\bigcup_{i=1}^sX_i$ separates every two vertices of $a_1,\ldots,a_p$
in $G'-X$, and thus also any two sets $A_i,A_j$ in $G-X$.

\new{We conclude that}
\begin{align*}
|X|&=|X'|+\sum_{i=1}^s|X_i| \leq k-1+\sum_{i=1}^s(p_i-1)(k-1)\\
&\leq k-1+(p-s)(k-1) \leq k-1+(p-2)(k-1)=(p-1)(k-1),
\end{align*}
since $s\geq 2$ by choice of $X'$.
\end{proof}

Let $A$ be a vertex set in a multigraph $G$, and let $k$ be a positive integer.
An edge set $X$  \emph{$k$-perfectly separates} $A$
if for every  $a,a'\in A$ with $a\not\sim_k a'$ in $G-X$, the vertices $a,a'$ 
lie in different components of $G-X$. 
This means, that two vertices either are not \new{in the same component} or there are at least $k$ edge-disjoint paths between them.

\begin{lemma}\label{lem:gatessep2}
Let $k\in \N$, and let $A$ be a vertex set in a multigraph~$G$.
Then there is a set $X\subseteq E(G)$ of size at most $(|A|-1)(k-1)$ that $k$-perfectly separates $A$.
\end{lemma}
\begin{proof}
We use induction on $|A|$.
Let $A_1,\ldots,A_p$ be a partition of $A$ induced by the 
equivalence classes of $\sim_k$.
If $p=1$, the statement trivially holds as $X=\es$
$k$-perfectly separates $A$.
In particular, this covers the case $|A|=1$.

Therefore, we may assume that $p\geq 2$.
We apply Lemma~\ref{lem:gatessep} to obtain a set $X'\subseteq E(G)$ of size at most $(p-1)(k-1)$ 
that separates $A_i$ from $A_j$ for all distinct $i,j\new{\in\{1,\ldots,p\}}$.
Denote for every $i\in \{1, \ldots, p\}$ by 
$G_i$  the union of components in $G-X'$ that contain a vertex in $A_i$, 
and observe that the $G_i$ are pairwise disjoint by choice of $X'$.
By induction,
there is a set $X_i\subseteq E(G_i)$ of size at most $(|A_i|-1)(k-1)$ that $k$-perfectly separates $A_i\cap V(G_i)$
 in $G_i$.
Thus, $X=X'\cup X_1 \cup \ldots \cup X_p$ $k$-perfectly separates~$A$.
Observe that
\begin{align*}
|X|\leq (p-1)(k-1) + (|A|-p)(k-1) = (|A|-1)(k-1),
\end{align*}
which completes the proof.
\end{proof}

\section{Proof of the main theorem}\label{sec:mainproof}

We start with a brief proof sketch. The key trick is to force a gap 
between short and long cycles: by induction, we can ensure that 
there are no intermediate cycles, cycles of length between~$\ell$ and~$10\ell$.
This forces a lot of structure. 
Repeatedly, we will argue that this or that property is satisfied
because otherwise we would find an intermediate cycle. 
\new{We found this trick in the article of Birmel\'e et al.~\cite{BBR07}.}

Throughout we  fix a frame $F$ \new{such that $\ds(F)$ is maximal.}
As every long cycle that is not contained in the frame contains at least one $F$-path,
it is necessary to find structure in the $F$-paths.
To this end, we group $F$-paths to \emph{hubs}.
The hubs together with parts of the frame $F$ form the \emph{hub closures},
which essentially partition the edge set of $G$. 
Informally, the hub closures are the largest
$2$-connected pieces that may contain cycles without also containing a cycle of~$F$.
 
From the absence of intermediate cycles we will deduce 
that no hub closure contains a long cycle. 
That means that every long cycle in some sense follows along a cycle in $F$
(without actually being contained in $F$). 
In particular, it traverses at least two (in fact, at least three) distinct hub closures \new{or it uses a path between two branch vertices of $F$}.
To define a candidate hitting 
set we therefore disconnect  hub closures when this is possible with 
few edges and when this cuts a connection between branch vertices of $F$.
The resulting edge set is either a true hitting set, 
or we will be able to piece together $k$ edge-disjoint long cycles
that all traverse well-connected hub closures in the same way.

\new{
We start with the proof of Theorem~\ref{thm:EPlongcycles}. It will take up 
the rest of this article. }

\begin{proof}[Proof of Theorem~\ref{thm:EPlongcycles}]
We define 
\[
f(k,\ell) = 
210k^2\log k + 10\ell (k-1).
\]
We prove by induction on $k$ that 
\begin{equation}\label{inductionclaim}
\begin{minipage}[c]{0.8\textwidth}\em
if a graph $G$ does not contain $k$ edge-disjoint long cycles,
then it contains an edge set $X$ of size at most $f(k,\ell)$ 
that meets every long cycle.  
\end{minipage}\ignorespacesafterend 
\end{equation} 

Clearly,~\eqref{inductionclaim} is true 
 for $k=1$ as either $G$ contains a long cycle or $X=\emptyset$ meets all long cycles in $G$.
We therefore assume that 
\begin{align} \label{noKcycles}
	\emtext{$k\geq 2$ and
that $G$ does not contain $k$ edge-disjoint long cycles.}
\end{align}

Suppose $G$ contains a long cycle $C$ of length at most $10\ell$.
As $G-E(C)$ contains at most $k-2$ edge-disjoint long cycles, 
by induction there is a hitting set $X'\subseteq E(G)\setminus E(C)$ for $G-E(C)$
of size at most $210(k-1)^2\log (k-1) + 10\ell (k-2)$. 
Observe that $X=E(C) \cup X'$ is a hitting set of $G$ such that
\begin{align*}
|X|&=|X'|+|E(C)|\le 210(k-1)^2\log (k-1) + 10\ell(k-2) + 10\ell \\
 &\le 210k^2\log k + 10\ell (k-1) =f(k,\ell).
\end{align*}
Thus, 
we may assume that 
\begin{equation}\label{mediumcyc}
\emtext{
every long cycle of $G$ has length more than $10\ell$.
}
\end{equation}

We may also assume that every edge of $G$ lies in a long cycle. 
Otherwise, if $e\in E(G)$ is not contained in any long cycle,
then every hitting set of $G-e$ is also a hitting set of $G$.

Suppose, $G$ is not $2$-connected; that is, $G$ contains several blocks.
Note that every cycle lies in exactly one block.
Since every edge belongs to at least one long cycle,
every block contains a long cycle.
Let $B$ be a block of $G$
and let $k'$ be the maximal integer
such that $B$ contains $k'$ edge-disjoint long cycles.
Hence $0<k'<k-1$, as \new{$G-E(B)$} contains at least one long cycle that is edge-disjoint from every cycle in $B$.
Observe that \new{$G-E(B)$} contains at most $k-k'-1<k\new{-1}$ edge-disjoint long cycles.
We apply our induction hypothesis to $B$ and \new{$G-E(B)$}
and obtain a hitting set $X_1\subseteq E(B)$ in $B$ of size at most $210(k'+1)^2\log(k'+1) + 10\ell k'\leq 210(k'+1)^2\log k + 10\ell k'$
and a hitting set $X_2\subseteq E(G)\setminus E(B)$ of size at most $210(k-k')^2\log k+10\ell (k-k'-1)$.
Trivially $X=X_1\cup X_2$ is a hitting set in $G$ such that
\begin{align*}
|X| &\leq210(k'+1)^2\log k  + 10\ell k'
+ 210(k-k')^2\log k+10\ell(k-k'-1) \\
&\le 210\log k \left(k'^2+2k' +1 + k^2-2kk'+k'^2\right) + 10\ell(k-1)\\
&= 210\log k \left(2k'(k'+1-k) + 1 +k^2\right) + 10\ell(k-1)\\
&\le 210 k^2 \log k + 10\ell(k-1) = f(k,\ell)
\end{align*}
as $2k'(k'+1-k)+1\le 0$ holds because of $k'<k-1$.
Thus, we can assume that 
\begin{align}
	\emtext{
		$G$ is $2$-connected.
	}\label{twoConn}
\end{align}

\medskip

We now choose a frame $F$ of $G$ such that \new{$\ds(F)$ is maximal} (and we may assume that $G$ contains at least one long cycle, which implies that a frame in $G$ exists), which 
we let be fixed throughout the whole proof. 
As $F$ only contains long cycles, \eqref{mediumcyc} implies that  
\begin{align}
	\emtext{
		the girth of $F$ is \new{larger} than $10\ell$.
	}\label{girthF}
\end{align}
Next, we investigate $G-F$ and how the components of $G-F$ attach to $F$.

\subsection{Bridges of the frame}
In  light of \eqref{noKcycles} and \eqref{mediumcyc}, Lemma~\ref{framelem} now states:
\begin{equation}\label{Fprops}
\begin{minipage}{0.8\textwidth}
\em 
$F$ is connected; $\ds(F)<84k\log k$; every $F$-path $Q=u\ldots v$ is short 
and $F$ contains a unique short $u$--$v$-path $P$.
\end{minipage}
\end{equation}

For any $F$-path $Q=u\ldots v$,
we call the unique short $u$--$v$-path in $F$ its \emph{shadow} and denote it by $S_Q$.

An $F$-\emph{bridge} of $G$ or simply a \emph{bridge} is either an edge in $E(G)\sm E(F)$ with its two endvertices in $V(F)$,
or a component $K$ of $G-F$ together with all its neighbours $N$ in $F$ and all edges of $G$ joining $K$ and $N$.
Equivalently, a bridge is the union of all $F$-paths
that form a component in the graph on the set of all $F$-paths
where two $F$-paths are adjacent if they share an internal vertex.
For an $F$-bridge $B$ of $G$,
we call the vertices in $B\cap F$ the \emph{feet} of $B$ (in \new{$F$}).
The \emph{shadow} $S_B$ of $B$ is the union of the shadows of all $F$-paths contained in~$B$.

\begin{claim}\label{claim:bridgeShadowTreeDiameter}
For every bridge $B$,
the shadow $S_B$ is a tree of diameter less than~$\ell$.
\end{claim}
\begin{proof}
As $B$ is connected, it contains an $x$--$y$-path $Q$ between any two of its feet~$x,y$. 
The shadow of this $F$-path $Q$ connects $x$ and $y$ in $S_B$. 
As all vertices in $S_B$ that are no feet
lie in the shadow of an $F$-path between two feet, we conclude that $S_B$ is connected.

Suppose that $S_B$ contains a cycle $C$. 
Since $C$ is contained in $F$, it follows that $C$ is a long cycle, which, 
in turn, implies $\len(C)\geq 10\ell$, by~\eqref{mediumcyc}.
Pick two vertices $r_1,r_2$ in $C$ at distance precisely~$2\ell$ in $C$, and let $R$ 
be the subpath of $C$ of length $2\ell$ between $r_1$ and $r_2$.

Why is $r_i$ in $S_B$? Because there is an  $F$-path $Q_i\subseteq B$ 
whose shadow $P_i$ contains~$r_i$. Denote by $x_i,y_i$ the endvertices of $P_i$, 
and observe that $P_i$ is a short path, by Lemma~\ref{framelem}~\eqref{uniqueShadow}.
By the same statement, there exists also a short $x_1$--$x_2$-path $S$ in the shadow of $B$. 

Since $P_1\cup P_2\cup R\cup S \subseteq S_B$ has at most $5\ell$ edges it cannot contain a long cycle, 
and because it is a subset of $F$ it cannot contain a short cycle. 
\new{Thus, $x_iP_ir_i$ or $y_iP_ir_i$ is internally disjoint from $R$ for $i=1,2$
(we may assume the first one).}
In particular, 
this means that $S=x_1P_1r_1Rr_2P_2x_2$, and thus that $R\subseteq S$.
This, however, is impossible since $S$ has length at most~$\ell$ but $R$ 
has length~$2\ell$.  

We deduce that $S_B$ is a tree.
By the definition of a shadow, every leaf of $S_B$ is a foot. 
As any two feet of $B$ are connected by an $F$-path, their distance in $S_B$ is short by Lemma~\ref{framelem}~\eqref{uniqueShadow}.
Thus, the diameter of $S_H$ is less than $\ell$.
\end{proof}

\subsection{Hubs}

We define a graph $\cG$ 
on the set of all bridges of $G$, where 
two bridges $B_1,B_2$ are adjacent if their shadows share a common edge.
A \emph{hub} is the union of all bridges in a 
component of $\cG$. Thus, a hub is a subgraph of $G$ consisting of all bridges
that form a component in $\cG$. We say that a bridge $B$ \emph{belongs} to a hub $H$
if $B\subseteq H$, that is, if $B$ is part of the component in $\cG$ that defines $H$.
For a hub $H$,
the \emph{shadow} $S_H$ of $H$ is the union of the shadows of all bridges in $H$.
\new{See Figure~\ref{vxhubfig} for an illustration.}

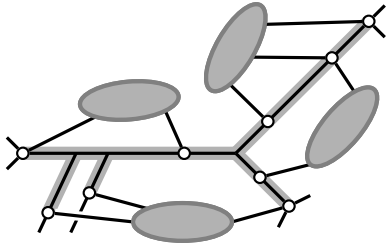
\begin{figure}[htb]
\centering
\begin{tikzpicture}[label distance=-0.5mm, scale=0.7]

\tikzstyle{bridge}=[shape=ellipse, outer sep=0mm, minimum height=5mm, minimum width=13mm,draw, ultra thick, color=dunkelgrau, fill=hellgrau]
\tikzstyle{shadow}=[hellgrau,line width=5pt]

\node[smallvx] (a) at (0,0){};
\coordinate (b) at (4,0);
\node[smallvx] (c) at (6.5,2.5){};
\node[smallvx] (d) at (5,-1){};
\coordinate (e) at (0.3,-1.5);

\coordinate (x) at (1,0);
\coordinate (y) at (1.6,0);
\coordinate (ey) at (0.9,-1.5);

\draw[hedge] (a) -- (b) node[smallvx,near end] (ab1){};
\draw[hedge] (b) -- (c) node[smallvx,near start] (bc1){} node[smallvx,near end] (bc2){};
\draw[hedge] (b) -- (d) node[smallvx,midway] (bd1){};
\draw[hedge] (x) -- (e) node[smallvx,near end] (e1){};
\draw[hedge] (y) -- (ey) node[smallvx,midway] (ey1){};

\draw[hedge] (a) -- ++(-0.3,-0.3);
\draw[hedge] (a) -- ++(-0.3,0.3);
\draw[hedge] (c) -- ++(0.3,-0.3);
\draw[hedge] (c) -- ++(0.3,0.3);
\draw[hedge] (d) -- ++(-0.2,-0.4);
\draw[hedge] (d) -- ++(0.4,0.2);

\node[bridge, rotate=60] (B1) at (4,2){};
\draw[hedge] (B1.south west) to (bc1);
\draw[hedge] (B1.south east) to (c);
\draw[hedge] (B1.south) to (bc2);
\node[bridge, rotate=60] at (4,2){};

\node[bridge, rotate=50] (B2) at (6,0.5){};
\draw[hedge] (B2.north east) to (bc2);
\draw[hedge] (B2.west) to (bd1);
\node[bridge, rotate=50] at (6,0.5){};

\node[bridge, rotate=0] (B3) at (3,-1.3){};
\draw[hedge] (B3.east) to (d);
\draw[hedge, double distance=1.2pt, white, double=black] (B3.west) to (e1);
\draw[hedge] (B3.north west) to (ey1);
\node[bridge, rotate=0] at (3,-1.3){};

\node[bridge, rotate=5] (B4) at (2,1){};
\draw[hedge] (B4.south east) to (ab1);
\draw[hedge] (B4.south west) to (a);
\node[bridge, rotate=5] at (2,1){};

\begin{scope}[on background layer]
\draw[shadow] (a.center) -- (b.center);
\draw[shadow] (b.center) -- (c.center);
\draw[shadow] (b.center) -- (d.center);
\draw[shadow] (x) to (e1.center);
\draw[shadow] (y) to (ey1.center);
\end{scope}

\end{tikzpicture}
\caption{A hub consisting of four bridges, and its shadow (in grey).}\label{vxhubfig}
\end{figure}

\new{
Before proceeding, we quickly note for later reference two basic properties of hub shadows 
that follow directly from the definition together with Claim~\ref{claim:bridgeShadowTreeDiameter}.
\begin{claim}\label{claim:shadowsEdgeDisjointAndConnected}
	The shadow  of any hub is connected and 
	the shadows of two distinct hubs are edge-disjoint.
\end{claim} }

We will write $\clos{H}$ for $H\cup S_H$ and call it the \emph{closure of $H$}.
\new{Hubs, their shadows and their closures constitute the key structure that captures 
how, in terms of long cycles, the rest of the graph attaches to the frame. 
In particular, when we will define the hitting set, in Section~\ref{hittingset}, 
we will exploit two features of hubs:
\begin{itemize}
\item hub closures do not contain long cycles (Claim~\ref{claim:noLongCycleHub}); and
\item every cycle that does not traverse only edges of a single hub closure is long (Claim~\ref{claim:cycleSeveralHubsLong}).
\end{itemize}
Except for basic properties (such as that every shadow of a hub is connected), 
these two are the only results about hubs that we need in the final part of the proof. 
The sole purpose of this subsection and the next is to prove Claims~\ref{claim:noLongCycleHub} 
and~\ref{claim:cycleSeveralHubsLong}. 
One way, therefore, to read this article would
be: jump directly to Section~\ref{hittingset} in order to see what role the claims play
in the finale of the proof and only afterwards come back here for the proofs of the two claims. 
}

\new{We start the path towards our first aim, Claim~\ref{claim:noLongCycleHub}, 
with a simple observation.}

\begin{claim}\label{claim:block}
For every hub $H$, the closure $\clos{H}$ is 2-connected.
\end{claim}
\begin{proof}
Since $G$ is $2$-connected, 
a bridge together with its shadow is $2$-connected, too.
The closure of a hub is the union of adjacent bridges together with their shadows. 
As adjacent bridges overlap on an edge, the union again is $2$-connected.
\end{proof}

For a hub $H$,
let $L_H$ be the graph with vertex set $E(S_H)$ and $e,f\in V(L_H)$ are adjacent in $L_H$ 
if $e,f$ share a common vertex in $G$ and there is a bridge $B$ which \new{belongs} to $H$ such that $e,f\in E(S_B)$.
Let $L_H^*$ arise from $L_H$ by adding all possible edges of the following type:
for all $e_1,\ldots,e_r\in V(L_H)$ sharing a common vertex in $G$
which induce a connected graph in $L_H$ add all edges $e_ie_j$ for distinct $i,j\in \{1,\ldots,r\}$.

\new{Where do the graphs $L_H$ and $L_{H^*}$ come from? The graph $L_H$ is a subgraph 
of the line graph on $S_H$ whose adjacencies encode how the shadows of the bridges of $H$ 
interact. The graph $L_{H^*}$ is obtained from $L_H$ by taking the transitive closure 
on each set of edges incident with the same vertex of $G$.}
\begin{claim}\label{LHconn}
For every hub $H$, the graph $L^*_H$ is connected. 
\end{claim}
\begin{proof} 
	We will prove that $L_H$ is connected which immediately proves the claim as $L_H\subseteq L^*_H$. 
	First, it is easy to see that for any bridge $B$ of $H$, the induced subgraph $L_H[E(S_B)]$
	on the edges of $S_B$ is connected.
	This holds as edges of $S_B$ with common endvertex in $G$
	are adjacent in $L_H$ as they belong to the shadow of the same bridge.
	The connectivity of $S_B$ \new{(Claim  \ref{claim:bridgeShadowTreeDiameter})} then implies the connectivity of $L_H[E(S_B)]$.
	
	Let $e,f\in V(L_H)$ be two edges of the hub $H$ that belong to the shadows of different bridges $B,B'$.
	The definition of hubs implies that there is a sequence of bridges $B=B_1, B_2, \ldots, B_r=B'$
	such that $S_{B_i}$ and $S_{B_{i+1}}$ share at least one edge. 
	As all $L_H[E(S_{B_i})]$ are connected in $L_H$, there is a path in $L_H$ joining $e$ and $f$.
\end{proof}


\new{The essence of the next claim is: provided some technical conditions are met, we can shortcut
any long path through a hub shadow to a short path by using parts of the hub. With the help of Claim~\ref{claim:nonAdjEdges}
we will then in Claim~\ref{claim:simpleVersion} get rid of the technical preconditions.  
}

\begin{claim}\label{claim:pathExtensionInShadow}
Let $H$ be a hub, and let $P$ be a path in $S_H$ 
such that every $P$-path in $F$ has length at least $3\ell$ 
and such that every pair of consecutive edges in $P$ is adjacent in $L_H^*$.
\new{Then there is a short path $\hat P$ between 
the endvertices of $P$ such that $\hat P\subseteq\clos H$
and  
such that $\dist_F(u,P)\le \ell$ holds  for every $u\in V(\hat P\cap F)$.}
\end{claim}

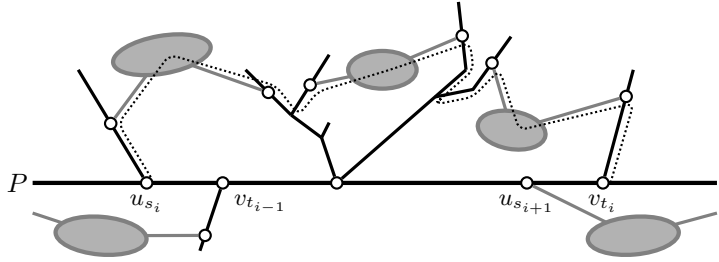
\begin{figure}[htb]
\centering
\begin{tikzpicture}[label distance=-0.5mm]

\tikzstyle{bridge}=[shape=ellipse, outer sep=0mm, minimum height=5mm, minimum width=13mm,draw, ultra thick, color=dunkelgrau, fill=hellgrau]

\draw[hedge,ultra thick] (0,0) -- (9,0);

\node at (-0.2,0) {$P$};

\node[smallvx,label=below:{\small $u_{s_i}$}] (a) at (1.5,0) {};

\draw[hedge] (a) -- ++ (-0.9,1.5) node[midway,smallvx] (a1){};

\node[smallvx,label=below right:{\small $v_{t_{i-1}}$}] (b) at (2.5,0){};

\draw[hedge] (b) -- ++(-0.3,-0.9) node[near end,smallvx] (b1){};

\node[smallvx] (c) at (4,0){};

\draw[hedge] (c) -- ++(-0.2,0.6) coordinate (c1);
\draw[hedge] (c1) -- ++(-0.4,0.3) coordinate (c2);
\draw[hedge] (c1) -- ++(0.1,0.2);

\draw[hedge] (c2) -- ++(0.5,0.8) node[midway,smallvx] (c4){};
\draw[hedge] (c2) -- ++(-0.6,0.6) node[midway,smallvx] (c5){};

\draw[hedge] (c) -- ++(1.7,1.5) coordinate[near end] (c7) coordinate (c3);
\draw[hedge] (c3) -- ++(-0.1,0.9) node[midway,smallvx] (c6){};

\draw[hedge] (c7) -- ++(0.5,0.1) -- ++(0.5,0.7) node[midway,smallvx] (c8){};

\node[smallvx,label=below:{\small $u_{s_{i+1}}$}] (d) at (6.5,0){};

\node[smallvx,label=below:{\small $v_{t_i}$}] (e) at (7.5,0){};

\draw[hedge] (e) -- ++(0.4,1.5) node[near end,smallvx] (e1){};

\node[bridge, rotate=10] (B1) at (1.7,1.7){};
\draw[hedge, dunkelgrau] (B1) edge (a1) edge (c5);

\node[bridge, minimum height=5mm, minimum width=9mm, rotate=0] (B2) at (4.6,1.5){};
\draw[hedge, dunkelgrau] (B2) edge (c4) edge (c6);

\node[bridge, minimum height=5mm, minimum width=9mm, rotate=-10] (B3) at (6.3,0.7){};
\draw[hedge, dunkelgrau] (B3) edge (c8) edge (e1);

\node[bridge, minimum height=5mm, minimum width=12mm, rotate=-5] (B0) at (0.9,-0.7){};
\draw[hedge, dunkelgrau] (B0) edge (b1) edge (0,-0.4);

\node[bridge, minimum height=5mm, minimum width=12mm, rotate=5] (B4) at (7.9,-0.7){};
\draw[hedge, dunkelgrau] (B4) edge (d) edge (9,-0.4);

\draw[thick,densely dotted,rounded corners=3pt,transform canvas={xshift=1mm}] (a) to (a1.center) to (B1.center) 
 to (c5.north) to (c2) to (c4.south) to (B2.center) to (c6.south) to (c3) to (c7)
to ++(-0.1,-0.1) to ++(0.5,0.1) to
(c8.south) to (B3.center) to (e1.south) to (e.center); 

\end{tikzpicture}
\caption{The path $Q_i$ (dotted).}\label{constructfig}
\end{figure}

\begin{proof}
As before, denote by $\leq_P$ the order on the vertices of $P$ induced by the path, where
we fix arbitrarily one of the two endvertices as first vertex. 

Denote by $P'$ the union of $E(P)$ and all edges in $F$ that have an endvertex in $P$. 
By assumption, the set $P'$ (seen as a vertex set in $L_H$) contains a path
in $L_H$ that contains $E(P)$ entirely 
(recall that two consecutive edges of $P$ may be nonadjacent in $L_H$, but adjacent in $L_H^*$). 
That means, there is a sequence of bridges $B_1,\ldots, B_t$ such that 
\begin{equation}\label{bridgecover}
\begin{minipage}[c]{0.8\textwidth}\em
$P\subseteq \bigcup_{i=1}^tS_{B_i}$, and $E(S_{B_i}\cap S_{B_{i+1}})\cap P'\neq\emptyset$
for $i=1,\ldots, t-1$.
\end{minipage}\ignorespacesafterend 
\end{equation} 

We choose the sequence $B_1,\ldots, B_t$ such that $t$ is minimal. Moreover, 
we \new{require} that the shadow of the first bridge $B_1$ contains the first edge of $P$ (and then 
the shadow of $B_t$ contains the last edge of $P$). To avoid double subscripts we 
write $S_i$ for the shadow $S_{B_i}$. 

We quickly note:
\begin{equation}\label{Psubpath}
\begin{minipage}[c]{0.85\textwidth}\em
for every bridge $B$, the intersection $S_B\cap P$ is a subpath of $P$.
\end{minipage}\ignorespacesafterend 
\end{equation} 
Indeed, this is the case as $S_B$ is connected and of diameter less than~$\ell$ (Claim~\ref{claim:bridgeShadowTreeDiameter}) and as there are 
no $P$-paths in $F$ of length at most~$3\ell$, by assumption.

We need a claim about the start and end of $P$:
\begin{equation}\label{startandend}
\begin{minipage}[c]{0.85\textwidth}\em
if $S_i$ contains the first vertex of $P$, then $i=1$, and if $S_i$ contains
the last vertex of $P$, then $i=t$.
\end{minipage}\ignorespacesafterend 
\end{equation} 
Suppose that $S_i$ contains the first vertex of $P$ and that $i>1$.
Then, omitting  the bridges $B_{1},\ldots, B_{i-1}$
we still have a sequence of bridges that satisfies~\eqref{bridgecover}; 
that $P$ is still contained 
in the union of the shadows is due to~\eqref{Psubpath}. But this contradicts the minimal choice 
of $B_1,\ldots,B_t$. The argument for the last vertex of $P$ is symmetric. 

We claim:
\begin{equation}\label{doubleshadow}
\begin{minipage}[c]{0.85\textwidth}\em
if $|i-j|>1$, then $S_i\cap S_j$ is either empty or consists
of a single vertex in $P$.
\end{minipage}\ignorespacesafterend 
\end{equation} 
Let $S_i\cap S_j$ be non-empty and $i<j-1$. 
Suppose first that $S_i$ and $S_j$ contain a common edge $e$ that lies in $P'$. 
Then we could omit the bridges $B_{i+1},\ldots, B_{j-1}$
from the sequence and still retain~\eqref{bridgecover}; that $P$ is still contained 
in the union of the shadows is due to~\eqref{Psubpath}. 

Next, suppose that $S_i\cap S_j$ contains a vertex $v$ outside $P$. Both shadows,
which are contained in $F$, contain a $v$--$P$-path
of length at most~$\ell$, by Claim~\ref{claim:bridgeShadowTreeDiameter}. As we had assumed
that there are no $P$-paths in $F$ of length at most~$3\ell$, this implies that $S_i\cap S_j$
contains a  $v$--$P$-path, which in turn means that $S_i\cap S_j$
contains an edge in~$P'$, which is impossible as we have seen. 
Thus, $S_i\cap S_j\subseteq P$. 

By~\eqref{Psubpath}, the set $S_i\cap S_j=S_i\cap S_j\cap P$
is a subpath of $P$. If it contains more than one vertex, it thus contains an edge in $P'$, which 
we had already excluded. This proves~\eqref{doubleshadow}.

For every $i=1,\ldots,t-1$ pick an edge $e_i$ in $S_i\cap S_{i+1}\cap P'$---this 
is possible, by~\eqref{bridgecover}.
Denote by $e_0$ the first edge of $P$, and by $e_t$ the last edge of $P$. 
For every $i=1,\ldots,t$, there is, by Claim~\ref{claim:bridgeShadowTreeDiameter}, a path 
in $S_i$ containing $e_{i-1}$ and $e_i$. Let $S'_i$ be a longest such path. 
By definition of a shadow, the endvertices of $S'_i$ are feet of $B_i$. 
Pick a path through $B_i$ and use it to complete $S'_i$ to a cycle $C_i$. 

We claim:
\begin{equation}\label{Ciprops}
\begin{minipage}[c]{0.8\textwidth}\em
\begin{enumerate}[\rm (i)]
\item $C_i\subseteq S_i\cup B_i$ is a short cycle;
\item there are vertices $u_i\leq_P v_i$ such that $u_iPv_i=C_i\cap P$; 
\item $C_i$ and $C_{i+1}$ meet in an edge of $P'$; and
\item $u_iPv_j\subseteq \bigcup_{s=i}^jS_s$ for every $1\leq i\leq j\leq t$.
\end{enumerate}
\end{minipage}\ignorespacesafterend 
\end{equation} 
That $C_i$ is short follows from~\eqref{Fprops}, Claim~\ref{claim:bridgeShadowTreeDiameter} and \eqref{mediumcyc};
(ii) follows from~\eqref{Psubpath}, and~(iii) holds since both $C_i$ and $C_{i+1}$ contain the edge~$e_i$.
Finally,~(iv) is a consequence of~(ii) and~(iii).

We also note that since $e_0\in E(C_1)$ and $e_t\in E(C_t)$:
\begin{equation}\label{firstandlast}
\emtext{
$u_1$ is the first vertex of $P$, and $v_t$ is its last.
}
\end{equation}

The intersections of $C_i\cap P=u_iPv_i$ are paths. Two such paths of consecutive 
cycles $C_i$ and $C_{i+1}$ may intersect in a single vertex or in a longer 
path (they  meet by~\eqref{Ciprops}~(iii)).  
Let $s_2<\ldots <s_{r}$ be precisely those indices such that 
$C_{s_i-1}\cap C_{s_i}\cap P$ contains at least one edge.
For a slightly less cumbersome notation, define also $t_{i-1}=s_{i}-1$
and set $s_1=1$ and $t_r=t$.  
Then the cycles $C_1,\ldots,C_t$ partition into sets $\{C_{s_i},\ldots, C_{t_i}\}$
for $i=1,\ldots, r$ such that always $C_{t_{i-1}}$ and $C_{s_i}$ share an edge of $P$. 
We claim:
\begin{equation}\label{Qijump}
\begin{minipage}[c]{0.8\textwidth}\em
for  $i=1,\ldots ,r$, there is a $P$-path $Q_i\subseteq \bigcup_{s=s_i}^{t_i}(B_s\cup S_s)$
between $u_{s_i}$ and $v_{t_i}$ such that $Q_i\cup u_{s_i}Pv_{t_i}$ is a short cycle.
\end{minipage}\ignorespacesafterend 
\end{equation} 
We prove this with Lemma~\ref{singlejumplem} and therefore check that the conditions of Lemma~\ref{singlejumplem} are satisfied. 
The first condition follows from (ii).
Why do $C_s$ and $C_{s+1}$ for $s\in\{s_i,\ldots, t_i-1\}$ meet outside $P$?
Because $C_s$ and $C_{s+1}$ have a common edge $e$ in $P'$ by~\eqref{Ciprops} that, however, 
cannot lie in $P$ by definition of the $s_i$. Thus, the endvertex of $e$ outside $P$ 
is a common vertex that lies outside $P$. The other endvertex of $e$, the one in $P$,
shows that $C_s$ and $C_{s+1}$ meet also in $P$. Now, the application of the lemma
yields a short cycle $C\subseteq \bigcup_{s=s_i}^{t_i}(B_s\cup S_s)$ such that
$C\cap P=u_{s_i}Pv_{t_i}$. As $S_{s_i}$ needs to contain an edge of $P$, by definition
of $s_i$, we deduce that $u_{s_i}<_P v_{t_i}$, and in particular that $u_{s_i}\neq v_{t_i}$. 
Deleting all vertices of $C$ in the interior
of $u_{s_i}Pv_{t_i}$ results in the desired $P$-path $Q_i$.

\medskip
We note \new{right away}:
\begin{equation}\label{closetoF}
\begin{minipage}[c]{0.8\textwidth}\em
every vertex of $F$ in $\bigcup_{i=1}^rQ_i$ \new{is at} distance at most~$\ell$ from $P$ in~$F$.
\end{minipage}\ignorespacesafterend 
\end{equation} 
Indeed, such a vertex in $F$ lies in some shadow $S_s$. Every such shadow meets $P$,
by~\eqref{bridgecover}, and has diameter at most~$\ell$ (Claim~\ref{claim:bridgeShadowTreeDiameter}),
which results in a distance at most~$\ell$ to $P$ in $F$ since $S_s\subseteq F$.

\medskip
Next \new{we claim}:
\begin{equation}\label{almost}
\emtext{
if $|i-j|>1$, then $Q_i$ and $Q_j$ are internally disjoint.
}
\end{equation}
Since two distinct bridges that meet meet in their shadows,
we obtain that $Q_i\cap Q_j$ is contained in
\[
\left(\bigcup_{s=s_i}^{t_i}S_s\right)
\cap 
\left(\bigcup_{s=s_j}^{t_j}S_s\right),
\] 
which is contained in $P$ by~\eqref{doubleshadow} as $|s_j-t_i|>1$ since $|j-i|>1$. 
Since $Q_i$ and $Q_j$ are $P$-paths, they can thus only meet in their endvertices. This proves \eqref{almost}. 

Next \new{we claim}:
\begin{equation}\label{orderorder}
u_{s_i} <_P u_{s_{i+1}}<_P v_{t_i} <_P v_{t_{i+1}} \emtext{ for }i=1,\ldots, r-1.
\end{equation}
We prove this by induction on $i$.
By definition of the $s_i$, the paths $u_{s_i}Pv_{t_i}$ and $u_{s_{i+1}}Pv_{t_{i+1}}$
have a common edge. This implies $u_{s_i}<_Pv_{t_{i+1}}$. 

Suppose that $u_{s_{i+1}}\leq_P u_{s_i}$. Then $i>1$, by~\eqref{startandend}
and~\eqref{firstandlast}. 
By induction, we get 
$u_{s_{i-1}}<_P u_{s_i}<_Pv_{t_{i-1}}$. 
Since we also have that  $u_{s_{i+1}}\leq_P u_{s_i}<_Pv_{t_{i+1}}$,
we deduce that $u_{s_{i-1}}Pv_{t_{i-1}}$ and $u_{s_{i+1}}Pv_{t_{i+1}}$ have
a common edge. By~\eqref{Ciprops}~(iv), this means that
there are $s\in\{s_{i-1},\ldots, t_{i-1}\}$ and $s'\in\{s_{i+1},\ldots t_{i+1}\}$
such that $S_s$ and $S_{s'}$ have an edge in common---but this contradicts~\eqref{doubleshadow}.
Thus, we get
\[
u_{s_i} <_P u_{s_{i+1}}<_P v_{t_i},
\] 
because $u_{s_i}Pv_{t_i}$ and $u_{s_{i+1}}Pv_{t_{i+1}}$ have a common edge.
Suppose that $v_{t_{i+1}}\leq_P v_{t_i}$. 
By~\eqref{startandend} and~\eqref{firstandlast}, 
this implies $t_{i+1}<t$, which in turn implies $i+1<r$.
Moreover, 
$u_{s_{i+1}}Pv_{t_{i+1}}\subseteq u_{s_i}Pv_{t_i}$. 
By definition of the $s_i$, it follows that $u_{s_i}Pv_{t_i}$ and $u_{s_{i+2}}Pv_{t_{i+2}}$
have an edge in common. Again from~\eqref{Ciprops}~(iv) we get that 
there is an $s\in\{s_{i},\ldots,t_{i}\}$
and an $s'\in\{s_{i+2},\ldots t_{i+2}\}$ such that $u_{s}Pv_s$ and $u_{s'}Pv_{s'}$
share an edge. Since this edge then lies in the shadow $S_s$ and in the shadow $S_{s'}$,
we obtain again a contradiction to~\eqref{doubleshadow}. This proves~\eqref{orderorder}.

\medskip

\new{We now apply Lemma~\ref{extintlem} to $Q_1,\ldots, Q_r$ in order to obtain
the desired path $\hat P$. 
We note that~\eqref{firstandlast}, \eqref{Qijump}, \eqref{almost}, and \eqref{orderorder}
ensure that all conditions of the lemma are satisfied. 
The resulting path $\hat P$ does not contain 
vertices $u$ of $F$ such that $\dist_F(u,P)> \ell$
because of \eqref{closetoF}.}
\end{proof}

\begin{claim}\label{claim:nonAdjEdges}
Let $H$ be a hub.
Let $P\subseteq F$ be a path of length \new{at least 2 and} at most~$5\ell$ with first and last edge $e$ and $f$ 
such that $e$ and $f$
belong to $S_H$ but are not adjacent in $L^*_H$.
Then there is no $e$--$f$-path in $L^*_H- (E(P)\sm\{e,f\})$.
\end{claim}
\begin{proof}
Suppose there is such a $P\subseteq F$ and an $e$--$f$-path $Q^*$ in $L^*_H- (E(P)\sm\{e,f\})$.
Among all such pairs $(P,Q^*)$ choose $P$ and $Q^*$ such that $\ell(Q^*)$ is minimal.
We claim that 
\begin{align}\label{eq:pFiveEll}
\emtext{$e$ and $f$ \new{do not share a common vertex} in  $G$.}
\end{align}
If $\ell(P)=5\ell$ then obviously $e$ and $f$ cannot have a vertex of $G$ in common. 
Suppose that $\ell(P)<5\ell$ and let $e'$ be the  successor of $e$ in $Q^*$ 
(note that $e'$ is a vertex in $Q^*$ but an edge in $G$).
We construct a path $P'$  that contradicts the minimal choice of $P$ together with $Q^*-e$.
As $\ell(P)<5\ell$, the graph $P+e'$ cannot contain a cycle because of $P+e' \subseteq F$ and~\eqref{girthF}.

If $P + e'$ is a path, set $P'=P+e'$.
Since $e'$ and $f$ \new{do not have a common vertex} in $G$ they are not \new{adjacent} in  $L^*_H$.
If $P+e'$ is not a path, set $P'=P-e+e'$.
If $e'$ and $f$ were adjacent in $L^*_H$, either $P+e'$ contained a cycle 
or $P=ef$ and $e,f,e'$ all share a common vertex---then, however, 
 the definition of $L^*_H$ implies that $e$ and~$f$ have to be adjacent in $L^*_H$, too,
which we have excluded.

As $\ell(P)<5\ell$,  the new path $P'$ satisfies $\ell(P')\le 5\ell$
and there is a path in $L^*_H-(E(P')\setminus\{e',f\})$ joining its endvertices $e'$ and $f$,
namely $Q^* - e$.
Thus, $(P',Q^*-e)$  contradicts the minimality of $Q^*$. This proves \eqref{eq:pFiveEll}.

Consider the  subgraph $Q$ of $G$ that consists of the edges $V(Q^*)$ and all incident vertices. 
We claim that $Q$ is a path. 
By the definition of $L^*_H$,
$Q$ is connected.
\new{Clearly, $Q$ is not a cycle by~\eqref{eq:pFiveEll} and the fact that $Q^*$ is a path.}
Thus, if $Q$ is not a path, it contains a vertex $v$ of degree at least~$3$.
Starting with $e$, let $e'$ be the first vertex of $Q^*$ that, seen as an edge in $G$, contains~$v$ as an endvertex
and let $f'$ be the last such vertex of~$Q^*$.
As $d_Q(v)\geq 3$, the edges $e'$ and $f'$ are not adjacent in $L^*_H$ as $Q^*$ was chosen minimal.
Note that the path $e'Q^*f'$ in $L^*_H$
is shorter than $Q^*$ as $\{e',f'\}\neq \{e,f\}$, by~\eqref{eq:pFiveEll}.
Thus, the path $P'=e'f'$ together with the path $e'Q^*f'$ in $L^*_H$
form a pair $(P',e'Q^*f')$ that contradicts the minimality of $Q^*$.
Therefore, $Q$ is a path in $G$.

Our next aim is to find a subpath $Q'\subseteq Q$
that satisfies the following two conditions:
\begin{equation}
	\emtext{every $Q'$-path in $F$ has length at least~$5\ell$; and} \label{eq:farInternalDist} 
\end{equation}
\begin{align}
	\emtext{there is a $Q'$-path $R\subseteq F$ between the endvertices of $Q'$
of length~$5\ell$. } \label{eq:farOuterDistance}
\end{align}
The set of those subpaths that satisfy \eqref{eq:farInternalDist} is nonempty, since every 
subpath of $Q$  of length, 
say, at most $\ell$ satisfies~\eqref{eq:farInternalDist}---recall 
that the girth of $F$ is larger than $10\ell$ by \eqref{girthF}.

Pick a longest subpath $S$ of $Q$ that satisfies~\eqref{eq:farInternalDist} in the role of $S=Q'$.
If $S$ also satisfies \eqref{eq:farOuterDistance}, we found the desired path.
Thus, we may assume that the shortest $S$-path $R\subseteq F$ between the endvertices
$u$ and $v$ of $S$ has length larger than~$5\ell$. 
Suppose that $u,v$ are precisely the endvertices of $Q$. Since $\ell(P)\leq 5\ell$,
either $P$ is a shorter $S$-path than $R$, which is impossible, or $P$ 
contains a $S$-path of length less than~$5\ell$, 
which violates~\eqref{eq:farInternalDist}. 
Therefore, at least one of  
$u,v$ is not an endvertex of $Q$; let this be~$u$.

Thus, $S$ can be extended  by the unique neighbour $u'$ of $u$
 in $V(Q)\setminus V(S)$ to  a path $S'\subseteq Q$. 
By the maximality of $\ell(S)$, the path $S'$ does not satisfy~\eqref{eq:farInternalDist}.
This is only possible if there is an $S'$-path $R'\subseteq F$ between $u'$ and some vertex $y\in V(S)$
that has length less than~$5\ell$. Since $R=uu'\new{R'}y$ is an $S$-path in $F$ it follows 
from~\eqref{eq:farInternalDist} that 
\[
5\ell\geq\ell(R)=\ell(R')+1\geq 5\ell-1+1=5\ell.
\]
Setting $Q'=uSy$ yields a subpath of $Q$ satisfying~\eqref{eq:farInternalDist} 
and~\eqref{eq:farOuterDistance}.

Let $x$ and $y$ be the endvertices of $Q'$ (and of $R$).
We check that the conditions of Claim~\ref{claim:pathExtensionInShadow} are satisfied by $Q'$.
As $Q'$ satisfies \eqref{eq:farInternalDist}, every $Q'$-path in $F$ has length at least $3\ell$.
The path $Q'$ is a subpath of $Q$ for which $E(Q)$ is a path in $L^*_H$, and thus $Q^*$ is also a path in $L^*_H$.
This implies the second condition of the claim.
\new{Thus, by Claim~\ref{claim:pathExtensionInShadow},
there is a short $x$--$y$-path $\hat Q$ contained in $\clos H$ that 
uses no vertex of $F$ at distance more than $\ell$ from $Q'$ measured in $F$.}

Denote by $R'$ the path obtained from $R$ by removing the first $\ell+1$ vertices
and the last $\ell+1$ vertices. \new{As $R\subseteq F$, this is also the case for $R'$.}
Note that $\len(R')\geq 3\ell-2\geq 2\ell$ as $\len(R)=5\ell$.
We claim that every vertex of $R'$ has distance more than $\ell$ from $Q'$ in $F$.
Suppose not. Then there exists a $Q'$--$R'$-path $P_1$ of length at most~$\ell$.
From the endvertex of $P_1$ in $R'$ pick a subpath $P_2$ of $R$ that ends in $x$ or in $y$
and has length at most~$3\ell$ which is possible as $\ell(R)=5\ell$. 
Since $P_1\neq P_2$ the union  $P_1\cup P_2$  
either contains a cycle or is a $Q'$-path. It cannot contain a cycle, since such 
a cycle would be contained in $F$ but would have length at most~$\ell(P_1)+\ell(P_2)\leq 4\ell$.
Thus, $P_1\cup P_2\subseteq F$ is a $Q'$-path of length at most~$4\ell$---this
contradicts~\eqref{eq:farInternalDist} and hence $\dist_F(R',Q')>\ell$.

\new{Since the path $\hat Q$} does not contain any vertex in $F$ at distance more than~$\ell$
measured in $F$, it follows that \new{$\hat Q$} is disjoint from $R'$. 
We extend $R'$ to a  subpath $R''$ of $R$
that is a \new{$\hat Q$}-path. Then, $R''$ has length 
\[
	2\ell \le \ell(R')\le \ell(R'')\le \ell(R)=5\ell.
\]
\new{Consequently, as $\hat Q$ is short the cycle contained in $\hat Q\cup R''$,
which contains all of $R''$, has length 
between $2\ell$ and 
$\ell(\hat Q)+\ell(R)\leq \ell+5\ell=6\ell$,} which is impossible by~\eqref{mediumcyc}.
Thus, there are no counterexamples to the claim.
\end{proof}

Using \new{Claim \ref{claim:nonAdjEdges}}, we show that \new{the} assumptions of Claim~\ref{claim:pathExtensionInShadow} are always satisfied 
and thus we obtain a simpler version of Claim~\ref{claim:pathExtensionInShadow}.

\begin{claim} \label{claim:simpleVersion}
	Let $H$ be a hub. Then
\begin{enumerate}[\rm (i)]
\item \label{simpleii}every $S_H$-path in $F$ has length at least $4\ell$; and
\item \label{item:simpleVersion}for every path $P\subseteq S_H$,
there is \new{a short path $\hat P$ between the endvertices of $P$ such that $\hat P\subseteq\clos H$
and  
such that $\dist_F(u,P)\le \ell$ holds  for every $u\in V(\hat P\cap F)$.}
\end{enumerate}
\end{claim}
\begin{proof} 
	
To see (i), suppose there is an $S_H$-path $P=u\ldots v$ in $F$ of length less than~$4\ell$.
Let $e,f\in E(S_H)$ be such that $e$ and $f$ contain $u$ and $v$, respectively.
The edges $e$ and $f$ cannot share an endvertex because then $F$ would contain a cycle $P+e+f$ of length less than $5\ell$
which contradicts \eqref{girthF}.
In particular, $e$ and $f$ are not adjacent in $L^*_H$.
Extend $P$ by these two edges and apply Claim~\ref{claim:nonAdjEdges}
to this path \new{(which has length at least 2 and at most $5\ell$)}
in order to obtain a contradiction to $L^*_H$ being connected (Claim~\ref{LHconn}).
	
	Statement (ii) is exactly the statement of Claim \ref{claim:pathExtensionInShadow} 
	without the assumptions that $P$ induces a path in $L_H^*$ 
	and that $P$-paths in $F$ have length at least \new{$3\ell$}.
	The second assumption is satisfied by (i).
	To prove that $P$ induces a path in $L_H^*$, 
	we show that every two distinct edges $e=uv,f=vw\in E(S_H)$ with a common endvertex $v$ are adjacent in~$L_H^*$.
	Assume the contrary. 
	Then $P=uvw$ is a path in $F$ of length~$2\le \new{2=\ell(P)\leq} 5\ell$.
	Applying Claim~\ref{claim:nonAdjEdges} to $P$, 
	we see that there is no $e$--$f$-path in $L_H^*$, which is impossible as $L_H^*$ is connected,
	by Claim~\ref{LHconn}.
\end{proof}

\subsection{\new{Cycles in hubs}}
\new{After collecting some basic (and partially technical) properties of hubs 
in the previous section,
we now come to those results about hubs, namely Claims~\ref{claim:noLongCycleHub}
and~\ref{claim:cycleSeveralHubsLong},
that will be needed in the last part of the proof when we define a hitting set.}

\begin{claim}\label{claim:noLongCycleHub}
For every hub $H$, the closure $\clos{H}$ of $H$ does not contain a long cycle.
Thus, the diameter of $\clos H$ is at most $\frac \ell 2$.
\end{claim}
\begin{proof}
\new{Using Claim~\ref{claim:simpleVersion} (ii) it is an easy observation that the diameter of $\clos{H}$ is at most $3\ell$:
indeed, let $x,y\in V(\clos{H})$
and observe that, by~\eqref{Fprops}, 
there is an $x$--$x'$-path and a $y$--$y'$-path in $\clos{H}$ of length at most $\ell$ such that $x',y'\in S_H$.
As $\dist_{\clos{H}}(x',y')\leq \ell$ by Claim~\ref{claim:simpleVersion} (ii), this proves the observation.}

\new{Assume for a contradiction that $\clos{H}$ contains a long cycle.
Let $C$ be a shortest long cycle in $\clos{H}$.
Thus $\ell(C)\geq 10\ell$ by~\eqref{mediumcyc}.
Consider  a $C$-path $P$ contained in~$\clos{H}$.
Then both paths in $C$ between the endvertices of $P$ are
at most as long as~$P$: indeed, otherwise the longer
of the two cycles in $C\cup P$ through $P$ is long, as $\ell(C)\geq 2\ell$,
but it is shorter than $C$, which contradicts the choice of $C$ as shortest long cycle in $\clos H$.}

\new{Recall that $\ell(C)\geq 10\ell$ and pick $x,y\in V(C)$ such that $\dist_C(x,y)\geq 5\ell$.
As there is an $x$--$y$-path $Q$ in $\clos{H}$ of length at most $3\ell$,
there is a $C$-path $P=u\ldots v$ in $\clos H$ 
such that $P$ is shorter than both $u$--$v$-paths in $C$, which is a contradiction.}

\new{This, in particular, implies that 
the diameter of $\clos{H}$ is at most $\frac{\ell}{2}$; recall that $\clos H$ is 
$2$-connected by Claim~\ref{claim:block}.}
\end{proof}

\new{Let us draw an easy consequence from  Claim~\ref{claim:noLongCycleHub}:
if $H$ is a hub, then its shadow $S_H$ cannot contain a cycle since any such cycle 
must be long, by~\eqref{Fprops} and $S_H\subseteq F$. Thus, it follows with 
Claim~\ref{claim:shadowsEdgeDisjointAndConnected} that:
}
\begin{claim}\label{claim:longCycleShadow}
For every hub $H$, the graph $S_H$ is a tree.
\end{claim}

\begin{claim}\label{claim:arbLength}
Let $H$ be a hub, and let $u,v\in V(S_H)$, $u\neq v$.
Let $r\in \N\cup \{0\}$ be such that
the unique $u$--$v$-path $Q\subseteq S_H$ has length
$$r\ell < \ell(Q) \le (r+1) \ell.$$
Then, for any $t\in\{0,\ldots,r\}$, there is a $u$--$v$-path $P\subseteq \clos{H}$ 
such that 
\[
t \ell \le \ell(P) \le t\ell+\new{2\ell}.
\]
\end{claim}
\begin{proof}
If $t=r$, then we can choose $P=Q$.
Suppose therefore that $t\in \{0,\ldots,r-1\}$. 
Note that $\ell(Q)> r\ell \ge (t+1)\ell$.
Let $x\in V(Q)$ be the vertex on $Q$ with $\ell(uQx)=(t+1)\ell$.
Next we use Claim~\ref{claim:simpleVersion}~\eqref{item:simpleVersion}
to obtain 
\new{a short $x$--$v$-path $R\subseteq\clos H$}
that uses no vertices of $F$ at distance more than $\ell$ from $xQv$ measured in $F$.

Starting from $u$, let $y$ be the first vertex of $Q$ that lies in $R$.
\new{We first note that $\dist_F(y,xQv)=\dist_{S_H}(y,xQv)$. Indeed, otherwise there is 
an $S_H$-path of length at most $\ell$ as $\dist_F(y,xQv)\leq \ell$ (because $y\in V(\cE)$)---this, however, contradicts Claim~\ref{claim:simpleVersion}~\eqref{simpleii}.
Then, $\dist_{Q}(y,xQv)=\dist_{S_H}(y,xQv)\leq\ell$ as $Q\subseteq S_H$ and $S_H$ is a tree by 
Claim~\ref{claim:longCycleShadow}, and 
therefore 
$t\ell \le \ell(uQy) \le \ell(uQx)=(t+1)\ell$.}
Thus the path $P=uQyRv$ is a path in $\clos{H}$ such that 
\[
t\ell \le \ell(uQy)\le \ell(P) \le \ell(uQx)+\ell(R) \le (t+1)\ell + \new{\ell}.
\]
\end{proof}

For a hub $H$, 
we call those vertices $v\in V(S_H)$ that have neighbours in $F- S_H$ the \emph{gates} of $H$. 
\new{Recall that all $S_H$-paths in $F$ have length at least $2$, by Claim~\ref{claim:simpleVersion}(i).}
Equivalently, $v$ is a gate of $H$ if it lies in $\clos H$ and has a neighbour outside $\clos H$.
Thus,
every path in $G$ that contains a vertex in $G- \clos{H}$ and a vertex in $\clos{H}$
also contains a gate of $H$.
\new{Gates will play a role when we define the hitting set.}

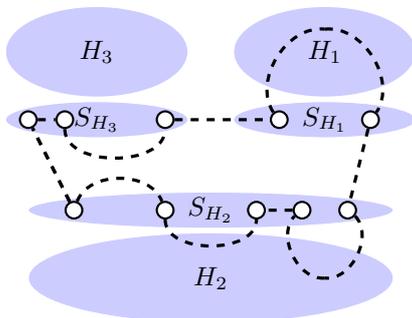
\begin{figure}
\centering
\begin{tikzpicture}[scale=0.6]
\draw[convcols] (0,0) ellipse (2 and 1);
	\node (h1) at (0,0){$H_3$};
\draw[convcols] (0,-1.5) ellipse (2 and 0.4);
	\node (sh1) at (0,-1.5){$S_{H_3}$};
\draw[convcols] (5,0) ellipse (2 and 1);
	\node (h2) at (5,0){$H_1$};
\draw[convcols] (5,-1.5) ellipse (2 and 0.4);
	\node (sh2) at (5,-1.5){$S_{H_1}$};

\draw[convcols] (2.5,-5) ellipse (4 and 1);
	\node (h3) at (2.5,-5){$H_2$};
\draw[convcols] (2.5,-3.5) ellipse (4 and 0.4);
	\node (sh3) at (2.5,-3.5){$S_{H_2}$};	

\coordinate (x) at (2.5,-3.5);
\node[hvertex] (a) at ($(x)+(-3,0)$){} ;
\node[hvertex] (b) at ($(x)+(-1,0)$){} ;
\node[hvertex] (c) at ($(x)+(1,0)$){} ;
\node[hvertex] (d) at ($(x)+(2,0)$){} ;
\node[hvertex] (e) at ($(x)+(3,0)$){} ;
\coordinate (de) at ($(x)+(2.5,-1.5)$) ;

\coordinate (y) at (0,-1.5);
\node[hvertex] (f) at ($(y)+(-1.5,0)$){} ;
\node[hvertex] (g) at ($(y)+(-0.7,0)$){} ;
\node[hvertex] (h) at ($(y)+(1.5,0)$){} ;

\coordinate (z) at (5,-1.5);
\node[hvertex] (i) at ($(z)+(-1,0)$){} ;
\node[hvertex] (j) at ($(z)+(1,0)$){} ;
\coordinate (ij) at ($(z)+(0,2)$) ;

\draw[hedge,dashed, bend left=70] (a) to (b);
\draw[hedge,dashed, bend right=90] (b) to (c);
\draw[hedge,dashed] (a) to (f);
\draw[hedge,dashed] (f) to (g);
\draw[hedge,dashed, bend right=90] (g) to (h);

\draw[hedge,dashed] (a) to (f);

\draw[hedge, dashed] (h) to (i);
\draw[hedge, dashed, bend left=65] (i) to (ij);
\draw[hedge, dashed, bend left=65] (ij) to (j);

\draw[hedge, dashed] (j) to (e);
\draw[hedge, dashed, bend left=70] (e) to (de);
\draw[hedge, dashed, bend left=70] (de) to (d);

\draw[hedge, dashed] (d) to (c);
	
\end{tikzpicture}
\caption{The dashed cycle traverses and visits $H_1$ once, it traverses $H_2$ once and visits $H_2$ twice.
	It does not traverse $H_3$, thus it also does not visit it.}
\label{fig:traversing}
\end{figure}

\begin{claim}\label{claim:cycleSeveralHubsLong}
Let $C$ be a cycle,
and let $H$ be a hub such that $C$ contains an edge both in $E(\clos{H})$ and in $E(G)\sm E(\clos{H})$.
Then $C$ is long.
\end{claim}
\begin{proof}
We say a cycle $C$ \emph{traverses} a hub  $H$ if $C$ contains an edge of $H$.
The number of traversals of $H$ is the number of components of $C\cap \clos H$ 
that  contain an edge of $H$. For hubs $H$ that are traversed by $C$,
we define the number of \emph{visits} as the number of components of $C\cap\clos H$
(which will be larger than the number of traversals if $C\cap\clos H$ has components
that are contained in the shadow of $H$). When $C$ fails to traverse $H$ then the 
number of visits is~$0$. \new{See also Figure~\ref{fig:traversing}.}

Suppose there is a short cycle that contains an edge of some hub closure $\clos H$
but is not completely contained in $\clos H$. Choose such a cycle $C$ such 
that the total number of hub traversals is minimal and subject to that 
choose $C$ such that 
the total number of visits is minimal.

We claim:
\begin{equation}\label{one}
\emtext{
if $C$ traverses a hub $H$, then $C\cap \clos H$ is a path.
}
\end{equation}
Suppose that $C\cap\clos H$ has a component $Q_1$ with an edge in $H$ (as $C$ traverses $H$)
and a second component (with or without edge in $H$). By Claim~\ref{claim:noLongCycleHub},
the diameter of $\clos H$ is at most~$\tfrac{\ell}{2}$. Thus, there is a $C$-path $P\subseteq\clos H$
of length at most~$\tfrac{\ell}{2}$ that starts in $Q_1$ and ends in another component $Q_2$ of $C\cap\clos H$.
Let $D_1,D_2$ be the two cycles in $C\cup P$ that contain $P$. 
We observe that 
\begin{equation}\label{two}
\emtext{
each of $D_1,D_2$ shares an edge with $\clos H$ but is not contained 
in $\clos H$. 
}
\end{equation}
Indeed, each of $D_1,D_2$ shares an edge with $\clos H$
because of $P\subseteq \clos H$. Neither of $D_1,D_2$ is contained in $\clos H$:
running along the $P$-path $D_i\cap C$ from the endvertex of $P$ in $Q_1$ we 
see that the first edge outside $Q_1$ lies also outside $\clos H$,
and there must be such an edge since $Q_1$ and $Q_2$ are distinct components.

Moreover,
\begin{equation}\label{three}
\emtext{
each of $D_1,D_2$ is a short cycle.
}
\end{equation}
For $i=1,2$, the length of~$D_i$ is at most $\ell(C)+\ell(P)\leq \ell+\tfrac{\ell}{2}$. As every long 
cycle has length at least~$10\ell$ by~\eqref{mediumcyc}, we deduce that $D_i$ is short.

The cycles $D_1,D_2$ are thus also counterexamples of the claim. To see that one of them contradicts
the minimal choice of $C$, we distinguish two cases.

First, assume that $C$ traverses a second hub $H'\neq H$. Then one of $D_1,D_2$, say $D_1$, meets an edge 
of $H'$. It follows that $D_2$ has at least one hub traversal less than $C$ and, in light of~\eqref{two} and~\eqref{three},
contradicts the minimality of $C$. Second, assume that $C$ traverses only one hub, namely $H$.
Then each of $D_1,D_2$ has fewer visits of $H$ (and at most the same number of traversals) 
and we again obtain a contradiction to the minimality of~$C$.
This proves~\eqref{one}.

Since $C$ is short, $C$ cannot be contained in the frame $F$ and therefore traverses a hub $H$. 
Then, by~\eqref{one}, the component $C\cap\clos H$ is a path, which we denote by~$Q_H$.
Its endvertices are two gates $g,g'$ of $H$.
If we replace $Q_H$ in $C$ by any $g$--$g'$-path in $\clos H$, 
we obtain a cycle, because otherwise $C\cap\clos H$ would have more than one component.	

Let $P_H$ be the (unique) $g$--$g'$-path in $S_H$, and
assume first that $\ell(P_H)<5\ell$. 
We replace in $C$ the path $Q_H$ by $P_H$ and obtain
a cycle $C'$ such that $\ell(C')\leq\ell(C)+\ell(P_H)\leq6\ell$.
Thus together with~\eqref{mediumcyc}, $C'$ is a short cycle. 
Moreover, $C'$ does not traverse $H$ anymore
as $C'\cap\clos H\subseteq S_H$.
Thus, $C'$ contradicts the minimal choice of $C$.

Second, assume that  $\ell(P_H)\ge 5\ell$.
By Claim~\ref{claim:arbLength}, 
there is a $g$--$g'$-path $P'_H$ in $\clos H$ with $\ell \le \ell(P'_H) \le 3\ell$.
Thus, if we replace $Q_H$ by $P'_H$ in $C$, we obtain a cycle $C'$ such that
$\ell\leq\ell(P)\leq\ell(C') \leq\ell(C)+\ell(P'_H)\le 4\ell$,
which is the final contradiction to~\eqref{mediumcyc}.
\end{proof}

\subsection{The hitting set} \label{hittingset}
We distinguish two cases: that $F$ is a cycle \new{and that it is not.}
\new{While it would be possible to treat both cases at once, we prefer not do so. 
The case when $F$ is a cycle is naturally much simpler but already contains some 
of the main ideas. We feel it might be instructive to see these ideas first in a
simpler setting.}
\begin{claim}\label{trivFclm}
Unless there is a hitting set of at most $k-1$ edges,
the frame $F$ is not a cycle.
\end{claim}
\begin{proof}
Assume $F$ to be a cycle. 
\new{If there are no hubs then $G=F$, and the graph becomes 
acyclic once a single edge is deleted. 
We may therefore assume that there are hubs.}

As shadows of hubs are trees \new{(Claim~\ref{claim:longCycleShadow}) 
and $F$ is a cycle, every shadow of a hub is a path.} 
In particular, 
the cycle $F$ cannot lie in a single shadow. 
\new{Let $u_1,u_2\in V(F)$ be the two (distinct) endvertices of a some hub shadow.
As hub shadows are edge-disjoint (by Claim \ref{claim:shadowsEdgeDisjointAndConnected}) 
it follows that neither of 
$u_1$ and $u_2$ lies in the interior of any hub shadow.}

Denote by $P_1$ and $P_2$ the two edge-disjoint $u_1$--$u_2$-paths
in $F$.
For $i=1,2$,
we let $\overline P_i$ be the union of $P_i$ and all hubs $H$ so that $S_H\subseteq P_i$.
Then
\[
G=\overline P_1\cup\overline P_2.
\]
Indeed, any edge $e$ of $F$ is contained in $P_1\cup P_2$. 
If $e\in E(G)\sm E(F)$, then $e$ lies in a hub, and every hub is contained in 
either $\overline P_1$ or in $\overline P_2$ as its shadow lies in either $P_1$
or in $P_2$. 

Since hub closures are blocks in $\overline P_i$---the endvertices 
of their shadow-paths are cutvertices in $\overline P_i$---it follows from Claim~\ref{claim:noLongCycleHub}
that
every long cycle contains 
an edge of $\overline P_1$ and an edge of $\overline P_2$.
More precisely,
every long cycle can be decomposed into two $u_1$--$u_2$-paths---one in each $\overline{P}_i$.

Suppose that for $i=1$ or for $i=2$, there is a set $X$ of at most $k-1$ edges
that separates $u_1$ from $u_2$ in $\overline P_i$. Then, $X$ meets every 
long cycle, since every such cycle contains a $u_1$--$u_2$-path in both $\overline P_1$
and $\overline P_2$.
This means that $X$ is a hitting set of size at most~$k-1$,
and we are done.

\new{If, on the other hand, $u_1$ and $u_2$ cannot be separated in either of $\clos P_1$
and $\clos P_2$ by fewer than $k$ edges then, by Menger's theorem,} for $i=1,2$ there are $k$ edge-disjoint $u_1$--$u_2$-paths
$Q_1^i,\ldots,Q^i_k$ contained in $\overline P_i$. We combine them 
to $k$ edge-disjoint cycles $Q_1^1\cup Q_1^2,\ldots, Q_k^1\cup Q_k^2$,
each of which is long, by Claim~\ref{claim:cycleSeveralHubsLong}, 
a contradiction to our assumption \eqref{noKcycles} that $G$ does not contain $k$ edge-disjoint long cycles.
\end{proof}

\new{As a consequence of Claim~\ref{trivFclm}, we
 may assume from now on that $F$ is not a cycle.
For a set of vertices $Z$, we call a $Z$-path or a cycle containing exactly one vertex of $Z$ a \emph{$Z$-ear}.
Bending the definition a bit, we use path notation for $Z$-ears $P$, even if $P$ is a cycle. 
In particular, we implicitly fix one orientation of $P$ if $P$ is a cycle, and 
then mean by $uPv$ the one subpath between $u$ and $v$ of $P$ which,
in the orientation of $P$, starts at $u$ and ends at $v$.
We may also say that a $Z$-ear $P$ has endvertices $u$ and $v$ (and we may have $u=v$) if 
$P\cap Z=\{u,v\}$.}

\new{
Recall that $U(F)$, which 
we abbreviate to $U$ in this subsection, denotes the set of vertices of $F$ of degree at least~$3$ in $F$, see 
Section~\ref{framesec}. 
As $F$ is not a cycle but connected, by~\eqref{Fprops}, with minimum degree
at least~$2$, it follows that 
$U\neq \emptyset$. This, in turn, implies that}
\begin{equation}\label{Upaths}
\emtext{
$F$ is the edge-disjoint union of $U$-\new{ears}.
}
\end{equation}

We distinguish two kinds of hubs:
A hub $H$ is a \emph{vertex-hub} if $S_H\cap U \neq \es$
and a \emph{path-hub} otherwise.
\new{As the shadow of a hub is connected by Claim \ref{claim:shadowsEdgeDisjointAndConnected}, we observe} that the shadow of a path-hub is completely
contained in some $U$-\new{ear} of $F$. 
Let $\cH$ be the set of all vertex-hubs.
A vertex-hub is shown in Figure~\ref{vxhubfig}, while the hub in Figure~\ref{pathhubfig} 
is a path-hub.

\begin{figure}[htb]
\centering
\begin{tikzpicture}[label distance=-0.5mm, scale=0.7]

\tikzstyle{bridge}=[shape=ellipse, outer sep=0mm, minimum height=5mm, minimum width=13mm,draw, ultra thick, color=dunkelgrau, fill=hellgrau]
\tikzstyle{shadow}=[hellgrau,line width=5pt]

\node[smallvx] (a) at (0,0){};
\node[smallvx] (z) at (10,0){};

\draw[hedge] (a) -- (z);

\node[smallvx] (b) at (1,0){};
\node[smallvx] (c) at (2,0){};
\node[smallvx] (d) at (2.5,0){};
\node[smallvx] (e) at (4.5,0){};
\node[smallvx] (f) at (5.5,0){};
\node[smallvx] (g) at (6.3,0){};
\node[smallvx] (h) at (7.5,0){};
\node[smallvx] (i) at (8.5,0){};

\node[smallvx] (x) at (3.3,0){};

\draw[hedge] (a) -- ++(-0.3,-0.3);
\draw[hedge] (a) -- ++(-0.3,0.3);
\draw[hedge] (z) -- ++(0.3,-0.3);
\draw[hedge] (z) -- ++(0.3,0.3);
\draw[hedge] (z) -- ++(0.4,0);

\node[bridge, rotate=0] (B1) at (2,1){};
\draw[hedge] (B1.south west) to (b);
\draw[hedge] (B1.south east) to (d);
\draw[hedge] (B1.south) to (d);

\node[bridge, rotate=0] (B2) at (3,-1){};
\draw[hedge] (B2.north west) to (c);
\draw[hedge] (B2.north east) to (f);
\draw[hedge] (B2.north) to (x);

\node[bridge, rotate=0] (B3) at (5.5,1){};
\draw[hedge] (B3.south west) to (e);
\draw[hedge] (B3.south east) to (h);

\node[bridge, rotate=0] (B4) at (7,-1){};
\draw[hedge] (B4.north west) to (g);
\draw[hedge] (B4.north east) to (i);
\draw[hedge] (B4.north) to (h);

\node at (-0.7,0) {$F$};

\begin{scope}[on background layer]
\draw[shadow] (b.center) -- (i.center);
\end{scope}

\end{tikzpicture}
\caption{A path-hub consisting of four bridges, and its shadow (in grey).}\label{pathhubfig}
\end{figure}
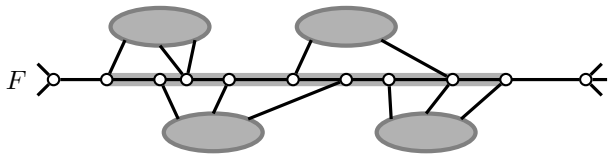

\new{Recall that gates are the vertices in a hub shadow that have neighbours outside the hub.}
For a hub $H$, let $A_H$ be the set of gates of $H$
and let $A^V=\bigcup_{H\in \cH}A_H$. 
\new{That is, $A^V$ is the 
set of all gates of vertex-hubs.}

Next, we give a bound from  above for $\sum_{H\in \cH} |A_H|$ for later use.
\new{As hub shadows are edge-disjoint by Claim~\ref{claim:shadowsEdgeDisjointAndConnected}, 
	every edge of $F$ incident with a gate $g\in A^V$ belongs to at most one hub shadow $S_H$.
	Hence, a gate $g$ belongs to at most $d_F(g)$ vertex-hubs.}

\new{
Let $P$ be a $U$-ear contained in $F$, and let $u,v$ be its endvertices. 
Assume that a vertex-hub $H$ has  a gate $g$ that lies in the interior of $P$.
Since hub shadows are connected, by  Claim~\ref{claim:shadowsEdgeDisjointAndConnected},
it follows that either $uPg$ or $gPv$ lies in $S_H$, say $uPg$. 
As a gate has, by definition, 
a neighbour outside the hub shadow, no inner vertex of $uPg$ can be a gate.
Since, moreover, hub shadows
are edge-disjoint, again by Claim~\ref{claim:shadowsEdgeDisjointAndConnected},
it follows that every $U$-ear in $F$ 
either contains two gates that each belong to only one vertex-hub, 
or the ear contains only one gate that belongs to at most two vertex-hubs, or it contains no gate.
Thus, $\sum_{H\in \cH} |A_H\sm U|$ is
at most twice the number of $U$-ears in $F$:
\begin{align*}
\sum_{H\in \cH} |A_H| & = \sum_{H\in \cH} |A_H\cap U|+\sum_{H\in \cH} |A_H\sm U|\\
&\le \sum_{g \in A^V\cap U} d_F(g) + 
2|\{P:\ \text{$P\subseteq F$ is a $U$-ear}\}| \\
&\leq \sum_{u\in U}d_F(u)+\sum_{u\in U}d_F(u)
\end{align*}
Recall that $\ds(F)$ is defined as $\sum_{u\in U}d_F(u)$, see Section~\ref{framesec}.
We thus obtain $\sum_{H\in \cH} |A_H|\leq 2\ds(F)$.
From~\eqref{Fprops} it therefore follows that
\begin{equation}\label{upperAV}
\sum_{H\in \cH} |A_H|\leq 168k\log k.
\end{equation}
}
\bigskip

Consider a $U$-\new{ear} $P$ of $F$. 
If the shadow of a vertex-hub, \new{which is connected by Claim~\ref{claim:shadowsEdgeDisjointAndConnected}}, \new{meets $P-U$},
then the intersection \new{of the shadow and $P$} is either a path containing at least one endvertex of $P$,
or the disjoint union of two paths each of which contains an endvertex of $P$ \new{(in the case when $P$ is a cycle recall that the shadow of a hub is a tree, by Claim~\ref{claim:longCycleShadow}).} 
Thus at most one component of $P-\bigcup_{H\in\cH}E(\clos H)$ is a path of length at least~$1$. 
If there is such a component $P'$, then let $u_P,v_P$ be the endvertices of $P'$.
Then $P'=u_PPv_P$. 
Let $\mathcal P$ denote the set of all $U$-\new{ears} $P$ of $F$ such that $P-\bigcup_{H\in\cH}E(\clos H)$
is not edgeless. 
We note that
\begin{equation*}
\emtext{if $P\in\mathcal P$, then $u_P,v_P\in A^V\cup U$}.
\end{equation*}
For $P\in\mathcal P$, we define $\clos P$ to be the 
union of $P'$ and all (path-)hubs $H$ so that $S_H\subseteq P'$.
\new{It is no coincidence that the notation $\clos P$ is similar to how 
we denote a hub closure. 
Indeed, vertex-hub closures and 
the subgraphs $\clos P$, $P\in\mathcal P$, have some similar properties (but 
also some that are dissimilar). In particular, the vertices $u_P$ and $v_P$
play the same role as the gates of a vertex-hub.
}

Next, we show
\begin{equation}\label{blubb}
\begin{minipage}[c]{0.8\textwidth}\em
for any two distinct $A,B\in\cH\cup\mathcal P$, the
graphs $\clos A$ and $\clos B$ are edge-disjoint and 
$\clos A\cap\clos B\subseteq A^V\cup \{u_P,v_P:P\in\mathcal P\}$. 
\end{minipage}\ignorespacesafterend 
\end{equation} 
Indeed, this follows directly if both $A,B\in\cH$, and also if both $A,B\in\mathcal P$,
since $U$-\new{ears} in $F$ meet only in $U$. 
If $A\in\cH$ and $B\in\mathcal P$, then 
$u_BBv_B$ meets $\bigcup_{H\in\cH}\clos H$ at most in $\{u_B,v_B\}$, by definition.

We claim that 
\begin{equation}\label{allchoppedup}
G=\bigcup_{H\in\cH}\clos H\cup\bigcup_{P\in\mathcal P}\clos P
\end{equation}
\new{
We remark that this is quite similar to the, admittedly simpler, 
edge-partition that appeared in Claim~\ref{trivFclm}.
To prove~\eqref{allchoppedup},} consider an edge $e\notin\bigcup_{H\in\cH}E(\clos H)$ of $G$.
Assume first that $e$ is contained in the closure of a path-hub $L$.
The shadow of $L$ then is contained in a $U$-\new{ear} $P$ of $F$, by~\eqref{Upaths}. 
Since 
the shadow of $L$ is edge-disjoint from $\bigcup_{H\in\cH}\clos H$ this implies 
that $P\in\mathcal P$. Then $e\in E(\clos L)\subseteq E(\clos P)$. Second, we have to consider the 
case when $e$ is an edge of $F$ that lies outside every hub shadow. Let $P$ be the
$U$-\new{ear} of $F$ containing~$e$. Again we see that $P\in\mathcal P$ and trivially $e$
is contained in $\clos P$. This proves~\eqref{allchoppedup}.

Next we show
\begin{equation}\label{longclosP}
\emtext{
for every $P\in\mathcal P$, every cycle contained in $\clos P$
is short. 
}
\end{equation}
The graph $\clos P$ is the edge-disjoint union of path-hub closures and edges in $F$ that 
lie outside every hub shadow. In particular, the path-hub closures contained in $\clos P$
are blocks in $\clos P$. Thus, any cycle contained in $\clos P$ lies completely in some 
path-hub closure, which only contains short cycles, by Claim~\ref{claim:noLongCycleHub}.

\medskip
We call $P\in\mathcal P$ \emph{thick} if there are at least $k$ edge-disjoint $u_P$--$v_P$-paths in $\clos{P}$,
and \emph{thin} otherwise. 
If $P$ is thin, then
there is a set $X_P\subseteq E(\clos{P})$ of at most $k-1$ edges separating $u_P$ and $v_P$ in $\clos{P}$, by Menger's theorem.
As part of the hitting set we define 
$X_p$ as the union of all $X_P$ where $P\in\mathcal P$ is thin.
By~\eqref{Fprops}, we obtain
\[
|X_p| = \sum_{P\in\mathcal P}|X_P|\leq 
k\cdot \tfrac{1}{2}\ds(F)\leq k \cdot 42k\log k.
\]
We note that~\eqref{longclosP} implies that
\begin{equation}\label{nothin}
\begin{minipage}[c]{0.8\textwidth}\em
in $G-X_p$ every long cycle is edge-disjoint from $\clos{P}$ for every thin $P\in\mathcal P$.
\end{minipage}\ignorespacesafterend 
\end{equation}

Consider $H\in \cH$.
Applying Lemma~\ref{lem:gatessep2} with $\clos{H}$ and $A_H$ playing the roles of $G$ and $A$,
we obtain a set $X_H$ of size at most $|A_H|k$ that $k$-perfectly separates $A_H$ in $\clos{H}$.
Let $X_v=\bigcup_{H\in \cH} X_H$.
With~\eqref{upperAV} we find that	
\begin{align*}
|X_v|\leq k \sum_{H\in \cH} |A_H| \leq k \cdot 168 k \log k = 168k^2\log k.
\end{align*}

We will show that $X=X_p\cup X_v$ is a hitting set for long cycles in $G$. 
We note first that 
\begin{align*}
|X| = |X_p|+ |X_v| \le 42k^2\log k + 168 k^2\log k =210k^2\log k \le f(k,\ell).
\end{align*}
Thus, if $X$ is indeed a hitting set then the induction hypothesis~\eqref{inductionclaim} is 
proved.

Let $\cJ$ be the set of all graphs $J$ such that either $J=\clos{P}$ for a thick $P\in\mathcal P$,
or such that $J$ is a component of $\clos{H}-X$  for some $H\in \cH$.

\goodbreak

\begin{claim}\label{lastclaim}$\rule{0cm}{0cm}$
\begin{enumerate}[\rm (i)]
\item\label{item:sim} \new{For any $J\in \cJ$, we have} $g\sim_k g'$ in $J$ for all $g,g'\in V(J)\cap (A^V\cup\{u_P,v_P:P\in\mathcal P\})$.
\item\label{item:distinctJJ} Distinct $J,J'\in\cJ$ are edge-disjoint, and
their intersection $J\cap J'$ lies in $A^V\cup\{u_P,v_P:P\in\mathcal P\}$.
\new{Furthermore, if $J,J'\subseteq \clos H$ for some $H\in \cH$, then $J\cap J'=\es$.}
\item\label{item:iii} Every long cycle in $G-X$ is entirely contained in $\bigcup_{J\in \cJ}J$
and no long cycle is contained in a single $J\in\cJ$.
\end{enumerate}
\end{claim}
\begin{proof}
\new{Let us first prove statement \eqref{item:sim} in the case that $J$ is a component of $\clos{H}-X$ for some $H\in \cH$.
The set $X$ separates $A_H$ $k$-perfectly,
i.e. two gates $g,g'\in A_H$  either belong to distinct components in $\clos{H}-X$ or 
$g\sim_k g'$ holds in $\clos{H}-X$ (in particular they belong to the same component of $\clos{H}-X$).
Thus, if $g,g'\in J$ for a component $J$ of $\clos{H}-X$, we conclude $g\sim_k g'$ in $J$.}
Suppose now that $J=\clos{P}$ for a thick $P\in \cP$.
Then $u_P\sim_k v_P$ as $P$ is thick and $\clos P$ is disjoint from~$X$. 

Observe that~\eqref{item:distinctJJ}
follows from~\eqref{blubb} as all $J\in \cJ$ are subgraphs of graphs in $\cH\cup \cP$
and two $J,J'\in \cJ$ that belong to the same vertex-hub $H$ are disjoint by definition as components of $\clos H -X$.

To see \eqref{item:iii},  	
consider a long cycle $C$.  Since $G$ is, by~\eqref{allchoppedup}, the union of vertex-hub closures
and all $\clos P$ for $P\in\mathcal P$, it follows that $G-X$ is contained in the union of all $J\in\cJ$ 
and all $\clos P$ for thin $P\in\mathcal P$. 
By~\eqref{nothin}, the cycle $C$ is edge-disjoint from every $\clos P$,
when $P\in\mathcal P$ is thin, which means that $C$ is contained in the union of all $J\in\cJ$. 
Finally, $C$ cannot be contained in any single $J\in\cJ$ as this is either a subgraph of a hub closure
(recall Claim~\ref{claim:noLongCycleHub}) or equal to $\clos P$ for some $P\in\mathcal P$ 
(recall~\eqref{longclosP}).
\end{proof}

\new{Suppose that $G-X$ contains a long cycle $C$. 
	Then $C$ decomposes by Claim~\ref{lastclaim}~\eqref{item:iii}
	into edge-disjoint non-trivial paths $P_1,\ldots, P_s$ such that each $P_i$ is contained in some $J_i\in\cJ$ and such that 
	$J_i$ and $J_{i+1}$ are distinct for every $i$  (where we put $J_{s+1}=J_1$).
We choose such a $C$ such that the number $s$ of these paths is minimal.
}

\new{Let $g_{i-1}$ and $g_i$ be the endvertices of $P_i$ for $i=1,\ldots, s$, where $g_s=g_0$, and observe that 
	every $g_i$ either lies in $A^V$
	or in $\{u_P,v_P\}$ for some thick $P\in\mathcal P$, by Claim~\ref{lastclaim}~\eqref{item:distinctJJ}.
	We claim 
	\begin{equation}\label{everyJonce}
	J_i\neq J_j\emtext{ for every }i\neq j,
	\end{equation} 
	which means that $C$ traverses no $J\in \cJ$ more than once. 
	Suppose that this is false, and choose $J\in \cJ$ such that $C\cap J$ has at least two components that contain an edge.
	As $J$ is connected, there is a $C$-path $Q\subseteq J$ that joins two such non-trivial components of $C\cap J$.
	Then $C\cup Q$ contains two cycles $D_1$ and $D_2$ that pass through $Q$. 
Both cycles are edge-disjoint from $X$.
	As $\ell(C)>10 \ell$ by \eqref{mediumcyc}, 
	at least one of $D_1$ and $D_2$, say $D_1$, has length at least $5\ell$, and is thus long.
	However, $D_1$ decomposes into fewer paths than $C$ as $D_1\cap J$ contains fewer non-trivial components than $C\cap J$ does.
	This contradicts the choice of $C$ and thus  proves~\eqref{everyJonce}.	
}

\new{By Claim~\ref{lastclaim}~\eqref{item:sim}, 
		there are $k$ edge-disjoint $g_{i-1}$--$g_{i}$-paths $P^1_i,\ldots,P^k_i$ in $J_i$ for every $i=1,\ldots, s$.
		By~\eqref{everyJonce}, 
		any two distinct paths $P_i^r$ and $P_j^t$ are edge-disjoint for $i,j\in\{1,\ldots, k\}$ and $r,t\in \{1,\ldots, s\}$.
	Concatenating, for each $r\in\{1,\ldots, k\}$, the paths $P^r_1,\ldots, P^r_s$ we obtain a closed walk $W_r$. 
	Any two walks $W_{r}$ and $W_t$ are edge-disjoint.
	Every vertex of $W_r$ has degree at least~$2$. This is trivial for the inner vertices of a path $P_i$.
	The endvertex $g_i$ of $P_i$ and $P_{i+1}$ is incident with an edge of $J_i$ and an edge of $J_{i+1}$.
	As $J_i$ and $J_{i+1}$ are edge-disjoint by Claim \ref{lastclaim} (ii), these two edges are distinct.
	Thus, $W_r$ has minimum degree at least~$2$ and therefore contains a cycle $C_r$.
	As $W_r\cap J_i$ is acyclic for every $J_i$ (it consists of the path $P_i$ together with some isolated vertices by \eqref{everyJonce}),
	the cycle $C_r$ uses edges from at least two distinct $J\in \cJ$, and then 
by Claim~\ref{lastclaim}~(ii), also from two distinct hub closures.
As a consequence of  Claim~\ref{claim:cycleSeveralHubsLong} the cycle $C_r$ is long.
	Summing up, we found $k$ edge-disjoint long cycles $C_1,\ldots, C_k$,
	which is the final contradiction to \eqref{noKcycles}.
	Therefore, the set $X$ is indeed a hitting set for the long cycles in $G$.
}
\end{proof}

\bibliographystyle{amsplain}
\bibliography{erdosposa}

\vfill

\small
\vskip2mm plus 1fill
\noindent
Version \today{}
\bigbreak

\noindent
Henning Bruhn
{\tt <henning.bruhn@uni-ulm.de>}\\
Matthias Heinlein
{\tt <matthias.heinlein@uni-ulm.de>}\\
Institut f\"ur Optimierung und Operations Research\\
Universit\"at Ulm, Ulm\\
Germany\\

\noindent
Felix Joos
{\tt <f.joos@bham.ac.uk>}\\
School of Mathematics\\ 
University of Birmingham, Birmingham\\
United Kingdom

\end{document}